\documentclass[a4paper,reqno,11pt]{amsart}
\usepackage{amsfonts,amsmath,amsthm,amssymb,stmaryrd}
\usepackage{mathrsfs}

\usepackage[left=1 in, right=1 in,top=1 in, bottom=1 in]{geometry}

\allowdisplaybreaks

\numberwithin{equation}{section}
\usepackage{color}

\newtheorem{theorem}{Theorem}[section]
\newtheorem{lemma}[theorem]{Lemma}
\newtheorem{proposition}[theorem]{Proposition}

\newtheorem{remark}[theorem]{Remark}

\theoremstyle{definition}

\newcommand{\bR}{\mathbb R}

\newcommand{\bT}{\mathbb T}

\newcommand{\Div}{\operatorname{div}}
\newcommand{\curl}{\operatorname{curl}}

\renewcommand{\epsilon}{\varepsilon}

\providecommand{\abs}[1]{\left\vert#1\right\vert}
\providecommand{\norm}[1]{\left\Vert#1\right\Vert}

\providecommand{\norm}[1]{\left\Vert#1\right\Vert}

\def\dt{\partial_t}

\def\pa{\partial}

\def\RRvert2{\right \vert\! \right\vert}
\def\Lvert3{\left \vert\!\left\vert\!\left\vert}
\def\Rvert3{\right \vert\!\right\vert\!\right\vert}

\def\nab{\nabla}

\def\dt{\partial_t}

\def\hal{\frac{1}{2}}

\def\ls{\lesssim}

\def\nabak{\nab_{\mathcal{A}^\kappa}}

\def\a{\mathcal{A}}

\def\i{\mathcal{I}}
\def\j{\mathcal{J}}

\def\fj1{\mathcal{J}^{-1}}

\def\bp{\partial_z}

\def\nak{\nabla_{\mathcal{A}^\kappa}}

\def\ak{{\mathcal{A}}^{\kappa}}

\def\fk{{\psi^\kappa}}

\def\Rk{R^{\kappa}}
\def\Zk{Z^{\kappa}}
\def\Jk{J^{\kappa}}
\def\Ek{E^{\kappa}}

\makeindex

\title[FBP of MHD]{Well-posedness of axially symmetric incompressible ideal magnetohydrodynamic equations with vacuum under the Rayleigh-Taylor sign condition}

\author[X. Gu]{Xumin Gu}
\address[X. Gu]{Department of Mathematics, Shanghai University of Finance and Economics, People's Republic of China}
\email{gu.xumin@shufe.edu.cn}

\begin{document}


\begin{abstract}
We consider a free boundary problem for the axially symmetric incompressible ideal magnetohydrodynamic equations that describes the motion of the plasma in vacuum. Both the plasma magnetic field and vacuum magnetic field are tangent along the plasma-vacuum interface. Moreover, the vacuum magnetic field is composed in a non-simply connected domain and hence is non-trivial. Under the Rayleigh-Taylor sign condition on the free surface, we prove the local well-posedness of the problem in Sobolev spaces. Furthermore, we also prove the local well-posdeness under a more general  ``stability'' assumption for the initial data, which provided that the Rayleigh-Taylor sign condition is satisfied at all those points of the initial interface where the non-collinearity condition fails.
\end{abstract}

\maketitle

\section{Introduction}
\subsection{Eulerian formulation}
In this paper, we consider the free boundary problem of the axially symmetric incompressible ideal MHD equations:
\begin{equation}
  \label{mhd}
\begin{cases}
\partial_tu^r + (u^r \partial_r+u^z\partial_z) u^r-\dfrac{(u^{\theta})^2}{r} + \partial_r (P+\dfrac{1}{2}\abs{B}^2) = (B^r\partial_r+B^z\partial_z)B^r -\dfrac{(B^{\theta})^2}{r}&\text{in}\ \  \Omega(t),\\
\partial_tu^{\theta} + (u^r \partial_r+u^z\partial_z) u^{\theta}+\dfrac{u^{\theta}u^r}{r} = (B^r\partial_r+B^z\partial_z)B^{\theta} +\dfrac{B^{\theta}B^r}{r}&\text{in}\ \  \Omega(t),\\
\partial_tu^z + (u^r \partial_r+u^z\partial_z) u^z+ \partial_z (P+\dfrac{1}{2}\abs{B}^2) = (B^r\partial_r+B^z\partial_z)B^z &\text{in}\ \ \Omega(t),\\
\partial_tB^r +  (u^r \partial_r+u^z\partial_z) B^r = (B^r\partial_r+B^z\partial_z) u^r &\text{in}\ \  \Omega(t),\\
\partial_tB^{\theta} +  (u^r \partial_r+u^z\partial_z) B^{\theta}+\dfrac{B^ru^{\theta}}{r}= (B^r\partial_r+B^z\partial_z) u^{\theta}+\dfrac{u^rB^{\theta}}{r} &\text{in}\ \ \Omega(t),\\
\partial_tB^z +  (u^r \partial_r+u^z\partial_z) B^z = (B^r\partial_r+B^z\partial_z) u^z &\text{in}\ \ \Omega(t),\\
\partial_r u^r+\dfrac{u^r}{r}+\partial_z u^z =0 &\text{in}\ \ \Omega(t),\\
\partial_r B^r+\dfrac{B^r}{r}+\partial_z B^z =0 &\text{in}\ \ \Omega(t).
\end{cases}
\end{equation}
In the equation \eqref{mhd}, $u(t,x)=u^r(t,r,z)e_r+u^{\theta}(t,r,z)e_{\theta}+u^z(t,r,z)e_z$ is the Eulerian or spatial velocity field, $B=B^r(t,r,z)e_r+B^{\theta}(t,r,z)e_{\theta}+B^z(t,r,z)e_z$ is the magnet field, and $P$ denotes the pressure function of the fluid which occupies the moving vessel domain:
$$\Omega(t): \{(x_1, x_2, z)| 0\leq r< r(z, t), z\in \mathbb T\}.$$
Here $e_r=(\cos\theta,-\sin\theta,0), e_{\theta}=(-\sin\theta,\cos\theta,0), e_z=(0,0,1), r=\sqrt{(x^1)^2+(x^2)^2}, z=x^3$, $\theta=\arctan{\frac{x^2}{x^1}}$.
We require the following boundary condition on the free surface $\Gamma(t):=\partial\Omega(t)$:
\begin{equation}
\label{kbd}
V(\Gamma(t))=u\cdot n \,\,\text{on}\,\, \Gamma(t)
\end{equation}
and
\begin{equation}
\label{hbd}
B\cdot n=0,\,\,P+\dfrac{1}{2}\abs{B}^2=\dfrac{1}{2}\dfrac{C^2(t)}{r^2}, \,\,\text{on}\,\,\Gamma(t).
\end{equation}
The equation \eqref{kbd} is called the kinematic boundary condition which states that the free surface $\Gamma(t)$ moves with the velocity of the fluid, where $V(\Gamma(t))$ denote the normal velocity of $\Gamma(t)$ and $n$ is the outward normal of the domain $\Omega^+(t)$. The first condition of the equation \eqref{hbd} means the fluid is perfect conductor. The second condition expresses the continuity of pressure on the free interface and $C(t)$ is given by
\begin{equation}
\label{cformula}
C(t)=C(0)e^{\int_0^tA(\tau)\,d\tau}, \,\,A(t)=\dfrac{\int_{\mathbb T}(u\cdot n)\sqrt{1+(\partial_z r(z,t))^2}\,dz}{\int_{\mathbb T}(\ln R_S-\ln r(z,t))\,dz},
\end{equation}
$R_S$ is a constant larger than $r(z,t)$.

The system \eqref{mhd}-\eqref{hbd} can be used to describe the motion of the plasma confined inside a rigid wall and isolated from it by vacuum, which is one of laboratory plasma model problems (see \cite[Chapter 4.6.1]{Go_04}). For the details of the derivation of this system, we refer the reader to \cite[Section 1]{Gu_17}. Briefly, in the vacuum region $\Omega_v(t)$, we
assume the pre-Maxwell equations:
\begin{equation}
\label{premax}
\begin{cases}
\curl \mathcal{B} =0,\quad \Div \mathcal{B} =0 &\text{in }\Omega_v(t),\\
\curl \mathcal{E} =-\dt \mathcal{B},\quad \Div \mathcal{E} =0&\text{in }\Omega_v(t).
\end{cases}
\end{equation}
In the equations \eqref{premax}, $\mathcal{B}$ and $\mathcal{E}$ denotes the magnetic and electric fields in vacuum, respectively.
%
The boundary conditions are
\begin{equation}
\label{B1}
\mathcal{B}\cdot \nu=0,\quad \mathcal{E}\times\nu=0, \quad\text{on } \partial\Omega_w
\end{equation}
and
\begin{equation}
\label{B2}
\mathcal{B}\cdot n=0, n\times\mathcal E=(u\cdot n) \mathcal B\quad \text{on } \Gamma(t),
\end{equation}
where $\partial\Omega_w$ is a perfect conducting wall. Then under the axially symmetric settings, with the first equation of \eqref{B1} and the first equation of \eqref{B2}, we can derive a formula for the vacuum magnet filed $\mathcal B=\mathcal B^{\theta}(r,z,t) e_{\theta}$:
\begin{equation}
\mathcal B^{\theta}=\dfrac{C(t)}{r}.
\label{vab}
\end{equation}
And by considering the elliptic system of the vacuum electronic field:
\begin{equation*}
\begin{cases}
\nabla\times \mathcal E=-\dfrac{C'(t)}{r},\,\,&\text{in}\,\,\Omega_v(t),\\
\nabla\cdot \mathcal E=0, \,\,&\text{in}\,\,\Omega_v(t),\\
n\times \mathcal E= (u\cdot n)\dfrac{C(t)}{r} \,\,&\text{on}\,\, \Gamma(t), \\\nu\times \mathcal E=0 \,\,&\text{on}\,\, \pa \Omega_w,
\end{cases}
\end{equation*}
we can derive the formula for $C(t)$: \eqref{cformula}.

On the other hand, the motion of the plasma is connected with the vacuum through the jump condition on the interface $\Gamma(t)$:
\begin{equation*}
  \left(\big(P  + \dfrac{1}{2}|B|^2\big)I  - B\otimes B\right)n=\left(  \dfrac{1}{2}|\mathcal{B}|^2I  -\mathcal{B}\otimes\mathcal{ B}\right)n\quad\text{on } \Gamma(t)
\end{equation*}
and
\begin{equation}
\label{ebd}
 (B-\mathcal{B})\cdot n  = 0, \quad\text{on } \Gamma(t).
\end{equation}
Hence the plasma-vacuum interface problem reduces to the free boundary problem \eqref{mhd}--\eqref{hbd}.
Our purpose of this paper is to establishing the local well-posedness for this problem.
\subsection{Lagrangian reformulation}
We tranform the Eulerian problem \eqref{mhd}-\eqref{hbd} on the moving domain $\Omega(t)$ to be one on the fixed domain $\Omega$ by the use of Lagrangian coordinates.
Let $x\in\Omega$ be the Lagrangian coordinate and $\eta(x,t)$ be the Eulerian coordinate, which means $\eta(x,t)\in\Omega(t)$ denote the "position" of the fluid particle $x$ at $t$. Thus,
 $$\partial_t\eta(x,t)=u(\eta(x,t),t)\,\,\text{for}\,\, t>0,\,\,\eta(x,0)=x.$$
Now we denote $R(x,t)=\sqrt{(\eta_1)^2+(\eta_2)^2}, \Theta(x,t)=\arctan \dfrac{\eta_2}{\eta_1}, Z(x,t)=\eta_3$, and $r=\sqrt{(x_1)^2+(x_2)^2}$, $\theta=\arctan\dfrac{x_2}{x_1}$, $z=x_3$. Then we can derive
\begin{equation}
	\label{rst}
	\begin{cases}
&\partial_tR=u^r\left(R(x,t),Z(x,t),t\right),\\
&\partial_t\Theta=\dfrac{1}{R(x,t)}u^{\theta}\left(R(x,t),Z(x,t),t\right),\\
&\partial_tZ=u^z\left(R(x,t),Z(x,t),t\right).
\end{cases}
\end{equation}
and
\begin{equation}
R(x,0)=r, \Theta(x,0)=\theta, Z(x,0)=z.
	\label{inz}
\end{equation}
From the first and third equation of \eqref{rst} and the initial data \eqref{inz}, we have
\begin{equation}
	\label{rzRz}
	R(x,t)=R(r,z,t), \quad Z(x,t)=Z(r,z,t).
\end{equation}
And for $\Theta$, we have \begin{equation}
	\label{dtheta}
	\Theta(r,\theta, z, t)=\theta+\int_0^t\dfrac{u^{\theta}(R,Z,\tau)}{R}\,d\tau=\theta+\hat\Theta(r,z,t).
\end{equation}
Now we denote
\begin{equation}
	\label{deflag}
	\begin{cases}
	&v^r(r,z,t)=u^r\left(R, Z, t\right), v^{\theta}(r,z,t)=u^{\theta}\left(R, Z, t\right), v^z(r,z,t)=u^z\left(R, Z, t\right),\\
	&b^r(r,z,t)=B^r\left(R, Z, t\right), b^{\theta}(r,z,t)=B^{\theta}\left(R, Z, t\right), b^z(r,z,t)=B^z\left(R, Z, t\right),\\&q(r,z,t)=P\left(R, Z, t\right)+\dfrac{1}{2}|B|^2\left(R, Z, t\right),
\end{cases}
\end{equation}
and denote the deformation tensor between $(R, Z)$ and $(r, z)$ as $\mathcal{F}_{ij}=\partial_{a^j}\zeta^i(r,z,t)$, where $\zeta=(R, Z), a=(r, z)$, (e.g. $\mathcal F_{11}=\partial_r R$).
Then we have the following Lagrangian version of \eqref{mhd} in the fixed reference domain $\Omega$:
\begin{equation}
	\label{eq:mhdo}
\begin{cases}
	\partial_t v^r-\dfrac{(v^{\theta})^2}{R}+\pa^\a_{R} q
=(b^r\pa^\a_R+b^z\pa^\a_Z)b^r-\dfrac{(b^{\theta})^2}{R}\,\,&\text{in}\,\,\Omega,\\
\partial_t v^{\theta}+\dfrac{v^{\theta}v^r}{R}
=(b^r\pa^\a_R+b^z\pa^\a_Z)b^{\theta}+\dfrac{b^rb^{\theta}}{R}\,\,&\text{in}\,\,\Omega,\\
\partial_t v^z+\pa^\a_Z q
=(b^r\pa^\a_R+b^z\pa^\a_Z)b^z\,\,&\text{in}\,\,\Omega,\\
\partial_tb^r =(b^r\pa^\a_R+b^z\pa^\a_Z) v^r\,\,&\text{in}\,\,\Omega,\\
\partial_tb^{\theta} +\dfrac{b^rv^{\theta}}{R}=(b^r\pa^\a_R+b^z\pa^\a_Z) v^{\theta}+\dfrac{v^rb^{\theta}}{R}\,\,&\text{in}\,\,\Omega,\\
\partial_tb^z =(b^r\pa^\a_R+b^z\pa^\a_Z) v^z\,\,&\text{in}\,\,\Omega,\\
\pa^\a_R(Rv^r)+\pa^\a_Z(Rv^z)=0\,\,&\text{in}\,\,\Omega,\\
\pa^\a_R(Rb^r)+\pa^\a_Z(Rb^z)=0\,\,&\text{in}\,\,\Omega,\\
 (f ,v, b)|_{t=0}=(\text{Id}, v_0, b_0).
\end{cases}
\end{equation}
where $\a=\mathcal{F}^{-T}$, $\pa^\a_{\zeta_i}:=\a_{ij}\partial_{a_j}$.
Here we donote
	\begin{equation}
		\Omega:=\{(x_1,x_2,x_3)|(x_1)^2+(x_2)^2< R_0, x_3\in \mathbb T\}
		\label{domain}
	\end{equation}
Two dynamic boundary conditions become:
\begin{equation*}
\begin{cases}
	q=\dfrac{1}{2}\dfrac{C^2(t)}{R^2}\,\,&\text{on}\,\,\Gamma\times(0,T],\\
(b^r,b^z) \a N =0\,\,&\text{on}\,\,\Gamma\times(0,T]\label{hebd}
\end{cases}
\end{equation*}
where $\Gamma:=\{(x_1,x_2, z)|(x_1)^2+(x_2)^2= R_0, z\in \mathbb T\}$, $N=(1,0)$ and
$$C(t)=C(0)e^{\int_0^t }A(\tau)\,d\tau,
	A(t)=\dfrac{\int_{\mathbb T}\left(v^r\partial_z Z(R_0,z,t)-v^z\partial_z R(R_0,z,t)\right)\,dz}{\int_{\mathbb T}\left(\ln R_S-\ln R(R_0,z,t)\right)\partial_z Z\,dz}.$$

By following the idea used in \cite{GW_16, Gu_17}, we can transfer the system \eqref{eq:mhdo} to a free-surface incompressible Euler system with a forcing term induced by the flow map. That is, by direct calculation, we have
\begin{equation*}
	\partial_t\left((b^r,b^z)\a\right)=0,
\end{equation*}
and hence
\begin{equation}
	\label{hformula}
	\begin{cases}
	b^r=(b_0^r\partial_r+b_0^z\partial_z) R,\\
	b^z=(b_0^r\partial_r+b_0^z\partial_z) Z.
 \end{cases}
\end{equation}
Then we plug \eqref{hformula} into the equation for $b^{\theta}$, we have
\begin{equation*}
	\partial_t b^{\theta}-\dfrac{v^rb^{\theta}}{R}=(b_0^r\partial_r+b_0^z\partial_z)v^{\theta}-\dfrac{v^{\theta}(b_0^r\partial_r+b_0^z\partial_z)R}{R}
\end{equation*}
Using \eqref{rst}, we have
\begin{equation*}
\partial_t\left(\dfrac{b^{\theta}}{R}\right)=\partial_t\left(\left(b_0^r\partial_r+b_0^z\partial_z \right)\Theta\right)
\end{equation*}
and
\begin{equation}
	\label{hfor2}
b^{\theta}=R(b_0^r\partial_r+b_0^z\partial_z)\Theta+\dfrac{Rb_0^{\theta}}{r}.
\end{equation}
Thus, with \eqref{hformula} and \eqref{hfor2}, we arrive at:
\begin{equation}
	\label{eq:mhd}
\begin{cases}
	\partial_t R= v^r\,\,&\text{in}\,\,\Omega,\\
		\partial_t Z = v^z\,\,&\text{in}\,\,\Omega,\\
\partial_t v^r-\dfrac{(v^{\theta})^2}{R}+\pa^\a_R q
=(b_0\cdot\nabla)^2R-R(b_0\cdot\nabla \Theta)^2\,\,&\text{in}\,\,\Omega,\\
\partial_t v^z+\pa^\a_Z q
=(b_0\cdot\nabla)^2Z\,\,&\text{in}\,\,\Omega,\\
\pa^\a_Rv^r+\dfrac{v^r}{R}+\pa^\a_Zv^z=0\,\,&\text{in}\,\,\Omega,\\
\partial_t \Theta = \dfrac{v^{\theta}}{R}\,\,&\text{in}\,\,\Omega,\\
\partial_t v^{\theta}+\dfrac{v^{\theta}v^r}{R}
=(b_0\cdot\nabla)(Rb_0\cdot\nabla\Theta)+b_0\cdot\nabla Rb_0\cdot\nabla\Theta\,\,&\text{in}\,\,\Omega,\\
(R,\Theta, Z, v)|_{t=0}=(r, \theta, z, u_0).
\end{cases}
\end{equation}
with boundary condition:
\begin{equation}
	q=\dfrac{C^2(t)}{R^2}\quad \text{on } \Gamma, C(t)=C(0)e^{\int_0^t }A(\tau)\,d\tau
	\label{qbdd}
\end{equation}
where $b_0\cdot\nabla=b_0^r\partial_r+b_0^z\partial_z+\dfrac{1}{r}b_0^{\theta}\partial_{\theta}$, $A(t)=\dfrac{\int_{\mathbb T}\left(v^r\partial_z Z-v^z\partial_z R(R_0,z,t)\right)\,dz}{\int_{\mathbb T}\left(\ln R_S-\ln R(R_0,z,t)\right)\partial_zZ\,dz}.$
In the system \eqref{eq:mhd}, the initial magnet field $b_0$ can be regarded as a parameter vector that satisfies
\begin{equation}
\label{bcond}
\partial_rb_0^r+\dfrac{1}{r}b_0^r+\partial_zb_0^z=0\,\,\text{in}\,\, \Omega \,\,\text{and}\,\, b_0^r=0  \,\,\text{on}\,\, \Gamma.
\end{equation}
\subsection{Previous works}
 Free boundary problems in fluid mechanics have important physical background and have been studied intensively in the mathematical community. There are a huge amount of mathematical works, and we only mention briefly some of them below that are closely related to the present work,  that is, those of the incompressible Euler equations and the related ideal MHD models.

For the incompressible Euler equations, the early works were focused on the irrotational fluids, which began
with the pioneering work of Nalimov \cite{N} of the local well-posedness for the small initial data and was generalized to the
general initial data by the breakthrough of Wu \cite{Wu1,Wu2} (see also Lannes \cite{Lannes}). For the irrotational inviscid fluids, certain dispersive effects can be used to establish the global well-posedness for the small initial data; we refer to  Wu \cite{Wu3,Wu4}, Germain, Masmoudi and Shatah \cite{GMS1, GMS2}, Ionescu and Pusateri \cite{IP,IP2} and Alazard
and Delort \cite{AD}. For the general incompressible Euler equations, the first local well-posedness in 3D was obtained by Lindblad \cite{Lindblad05} for the case without surface tension (see Christodoulou and Lindblad \cite{CL_00} for the a priori estimates) and
by Coutand and Shkoller \cite{CS07} for the case with (and without) surface tension. We also refer to the results of Shatah and Zeng \cite{SZ} and Zhang and Zhang
\cite{ZZ}. Recently, the well-posedness in conormal Sobolev spaces can be found by the the inviscid limit of the free-surface incompressible Navier-Stokes equations, see Masmoudi and Rousset \cite{MasRou} and  Wang and Xin \cite{Wang_15}.

The study of free boundary problems for the ideal MHD models seems far from being complete; it attracts many research interests, but up to now only few well-posedness theory for the nonlinear problem could be found. For the plasma-vacuum interface model that a surface current $J$ is added as an outer force term to the vacuum pre-Maxwell system \eqref{premax}, with the non-collinearity condition holding for two magnet fields on the boundary:
\begin{equation}
\label{st}
\abs{B\times \mathcal B}>0\,\,\text{on}\,\,\Gamma(t),
\end{equation} the well-posedness
of the nonlinear compressible problem was proved in Secchi and Trakhinin \cite{Secchi_14} by
the Nash-Moser iteration based on the previous results on the linearized problem \cite{Trak_10,Secchi_13}. The well-posedness of the linearized incompressible problem was proved by Morando, Trakhinin and Trebeschi \cite{Mo_14}, the nonlinear incompressible problem was sloved by Sun, Wang and Zhang \cite{Sun_17} very recently. In \cite{Gu_17}, instead of adding surface current $J$, the author considered an axially symmetric case of ideal MHD model that the vacuum magnet field is also non-trivial and the author established local well-posedness under the non-collinearity condition. On the other hand, Hao and Luo \cite{Hao_13} established
a priori estimates for the incompressible plasma-vacuum interface problem under the Rayleigh-Taylor sign condition:
\begin{equation}
\label{taylor}
\dfrac{\partial\left(P+\dfrac{1}{2}\abs{B}^2-\dfrac{1}{2}\abs{\mathcal B}^2\right)}{\partial{n}}\leq -\epsilon<0\,\,\text{on}\,\,\Gamma(t),
\end{equation}
under the assumption that the strength of the magnetic field is constant on the
free surface by adopting a geometrical point of view \cite{CL_00}. Recently, Gu and Wang proved the well-posedness of the incompressible plasma-vacuum problem under \eqref{taylor} with the vacuum magnet field is zero. In this paper, the author would establish the well-posedness of the plasma-vacuum interface problem under \eqref{taylor} in an axially symmetric setting. Furthermore, the author would also prove the local well-posdeness under a more general  ``stability'' assumption for the initial data, which provided that the Rayleigh-Taylor sign condition is satisfied at all those points of the initial interface where the non-collinearity condition fails. Finally, we also mention some works about the current-vortex sheet problem, which describes a velocity and magnet field discontinuity in two ideal MHD flows. The nonlinear stability of compressible current-vortex sheets was solved independently by Chen and Wang \cite{Chen_08} and Trakinin \cite{Trak_09} by using the Nash-Moser iteration. For incompressible current-vortex sheets, Coulombel, Morando, Secchi and Trebeschi \cite{CMST} proved an a priori estimate for the nonlinear problem under a strong stability condition, and  Sun, Wang and Zhang \cite{Sun_15} solved the nonlinear stability.

\section{Main Result}
Before stating our results of this paper, we may refer the readers to our notations and conveniences in Section \ref{Notation}.

We define the higher order energy functional
\begin{equation}
\label{edef}
\mathfrak{E}(t)=\norm{(v^r, v^{\theta}, v^z)}_4^2+\norm{(R, Z)}_4^2+\norm{(b_0\cdot\nabla R, Rb_0\cdot\nabla \Theta, b_0\cdot\nabla Z)}_4^2
\end{equation}
Then the main results in this paper are stated as follows.
\begin{theorem}\label{mainthm}
Suppose that the initial data $ (v_0^r, v_0^{\theta}, v_0^z) \in H^4_{r,z}(\Omega)$ with $v^r_0(0,z)=v^{\theta}_0(0,z)=0$ and $\partial_rv_0^r+\frac{1}{r}v_0^r+\partial_z v^z_0=0$, $(b_0^r, b_0^{\theta}, b_0^z) \in H^4_{r,z}(\Omega)$ with $b^r_0(0,z)=b^{\theta}_0(0,z)=0$ and $(b_0^r, b_0^{\theta}, b_0^z)$ satisfies \eqref{bcond} and that
\begin{equation}
\label{stc}
\partial_r\left(q_0-\dfrac{1}{2}\abs{\dfrac{C(0)}{r}}^2\right) \leq -\lambda<0\,\,\text{on}\,\,\Gamma
\end{equation}
holds initially. Then there exists a $T_0>0$ and a unique solution $(v^r, v^{\theta}, v^z, q, R, \Theta, Z)$ to  \eqref{eq:mhd} on the time interval $[0, T_0]$ which satisfies
\begin{equation}\label{enesti}
\sup_{t \in [0,T_0]} \mathfrak E(t) \leq P\left(\norm{\left(v_0^r, v_0^{\theta}, v_0^z\right)}_4^2+\norm{\left(b_0^r, b_0^{\theta}, b_0^z\right)}_4^2\right),
\end{equation}
where $P$ is a generic polynomial.

\end{theorem}
\begin{theorem}\label{mainthm2}
Suppose that the initial data $ (v_0^r, v_0^{\theta}, v_0^z) \in H^4_{r,z}(\Omega)$ with $v^r_0(0,z)=v^{\theta}_0(0,z)=0$ and $\partial_rv_0^r+\frac{1}{r}v_0^r+\partial_z v^z_0=0$, $(b_0^r, b_0^{\theta}, b_0^z) \in H^4_{r,z}(\Omega)$ with $b^r_0(0,z)=b^{\theta}_0(0,z)=0$ and $(b_0^r, b_0^{\theta}, b_0^z)$ satisfies \eqref{bcond}. Denote the set $\gamma$ as
$$\gamma:=\{(R_0, z)|b_0^z(R_0,z)=0\}\subset \Gamma.$$
If that
\begin{equation}
\label{ts}
\partial_r\left(q_0-\dfrac{1}{2}\abs{\dfrac{C(0)}{r}}^2\right) \leq -\lambda<0\,\,\text{on}\,\,\gamma
\end{equation}
holds initially. Then there exists a $T_0>0$ and a unique solution $(v^r, v^{\theta}, v^z, q, R, \Theta, Z)$ to  \eqref{eq:mhd} on the time interval $[0, T_0]$ which satisfies
\begin{equation}\label{enesti2}
\sup_{t \in [0,T_0]} \mathfrak E(t) \leq P\left(\norm{\left(v_0^r, v_0^{\theta}, v_0^z\right)}_4^2+\norm{\left(b_0^r, b_0^{\theta}, b_0^z\right)}_4^2\right),
\end{equation}
where $P$ is a generic polynomial.
\end{theorem}
\begin{remark}
Recall the vacuum magnet field formula \eqref{vab},  the condition $\abs{b_0^z}\geq \delta>0$ is actually the non-collinearity condition \eqref{st} under the axially symmetric settings. Thus, under the axially symmetric settings, Theorem \ref{mainthm2} establishes the local well-posedness of the equation \eqref{eq:mhd} under a more general  ``stability'' assumption for the initial data, which provided that the Rayleigh-Taylor sign condition is satisfied at all those points of the initial interface where the non-collinearity condition fails.
\end{remark}
\subsection{Strategy of the proof}
The strategy of proving the local well-posedness for the inviscid free boundary problems consists of three main parts: the a priori estimates in certain energy functional spaces, a suitable approximate problem which is asymptotically consistent with the a priori estimates, and the construction of solutions to the approximate problem.  For the incompressible MHD equations \eqref{eq:mhd}, we derive our a priori estimates in the following way. First, we divide \eqref{eq:mhd} into two sub-systems: one is for $(v^r, v^z, q, R, Z)$, the other one is for $(v^{\theta}, \Theta)$ (see \eqref{sub1} and \eqref{sub2}). The a priori estimates for $(v^{\theta}, \Theta)$ can be obtained by standard energy method. This is because there is no pressure in this subsystem and no boundary integral needs to be considered. Here, one will meet the difficulty to deal with the singularity brought by the cylinder coordinates, i.e. the estimates of $\dfrac{v^{\theta}}{r}$. Hence, we will apply the high order Hardy inequality to control these terms. On the other hand, the estimates for $(v^r, v^z, q, R, Z)$ is more complicated and it is the main part of this paper. We shall use tangential energy estimates combining with divergence and curl estimates to close the a priori estimates of $(\nu, q, \zeta)$, where we denote $\nu=(v^r, v^z), \zeta=(R, Z)$. During this process, there are several difficulties to deal with. In the usual derivation of the a priori tangential energy estimates of \eqref{sub1} in the $H^4_{r,z}$ setting, one deduces
\begin{equation}
\label{tp}
\begin{split}
&\hal\dfrac{d}{dt} \int_{\Omega}\abs{\bp^4 \nu}^2+\abs{\bp^4 ( b_0\cdot\nabla\zeta)}^2+\underbrace{\int_{\Gamma}-\a_{mj}\partial_{a_j} q\bp^4{\zeta}_m\bp^4 \nu_i{\mathcal A}_{i1}+\bp^4q\bp^4\nu_i\a_{i1}}_{\mathcal{I}_b}
\\&\quad \approx \underbrace{\int_\Omega \bp^4 D\zeta \bp^4  \zeta+ \bp^4 \nabla\zeta\bp^4  q}_{\mathcal{R}_Q}+l.o.t.,
\end{split}
\end{equation}
The first difficulty one will meet is the loss of derivatives in estimating $\mathcal{R}_Q$ (by recalling the energy functional $\mathfrak{E}(t)$ defined by \eqref{edef}).  Our idea to overcome this difficulty is, motivated by \cite{MasRou,Wang_15, GW_16}, to use Alinac's good unknowns $
\mathcal{V} =\bp^4 \nu- \bp^4\zeta \cdot\nabla_\a \nu$ and $ \mathcal{Q} =\bp^4 q- \bp^4\zeta \cdot\nabla_\a q$, which derives a crucial cancellation observed by Alinhac \cite{Alinhac}, i.e., when considering the equations for $\mathcal{V} $ and $Q$, the term $\mathcal{R}_Q$ disappears. The second difficulty is to estimate the boundary integral $\mathcal{I}_b$. Recalling the boundary condition for $q$, there are two possible ways to control it. One is to control $\abs{\bp^4\zeta}_{1/2}$ and use $(H^{-1/2}, H^{1/2})$ dual estimate. This requirement of the boundary regularity can actually be obtained by the non-collinearity condition \eqref{stc} (see \cite{Gu_17}). The another possible way is that, if the integral $I_b$ has symmetric structure, one can  use the Rayleigh-Taylor condition and obtain the control of $\abs{\bp^4\zeta}_{1/2}$. If the vacuum magnet field is trivial, then $q=0$ on $\Gamma$, the symmetric structure is somehow easy to check, see \cite{CL_00, GW_16}. For our system \eqref{eq:mhd}, $q$ does not vanish on the boundary, the symmetric structure is not clear at the first glance. Fortunately, with careful calculation and the definition of $\a, J$, we can find the symmetric structure successfully. Briefly, since $q=\hal\frac{C^2}{R^2}$ on $\Gamma$, then we have
\begin{equation*}
\begin{split}
I_b&=\int_{\bT}-\a_{mj}\partial_{a_j} q\bp^4{\zeta}_m\bp^4 \nu_i{\mathcal A}_{i1}+\bp^4q\bp^4\nu_i\a_{i1}\\&= \int_{\mathbb T} -R_0\dfrac{C^2}{R^3}\left(\bp^4R-\bp^4\zeta_j\ak_{j2}\partial_zR\right)\a_{i1}\nu_i\,dz-\int_{\bT}R_0\bp^4\zeta_m \ak_{m1}\partial_r q \a_{i1}\nu_i\,dz+l.o.t.
\end{split}
\end{equation*}
We also have
 \begin{equation*}
\begin{split}
\bp^4 R-\bp^4\zeta_j\a_{j2}\partial_z R&=\bp^4 R\left(1-\a_{12}\partial_zR\right)+\bp^4Z(-\a_{22}\partial_zR)
\\&=\bp^4 R\left((J)^{-1}J+(J)^{-1}\partial_rZ\partial_zR\right)+\bp^4 Z(\a_{22}J\a_{21})\\&=\bp^4 R\left((J)^{-1}\partial_zZ\partial_rR\right)+\bp^4 Z(\a_{22}J\a_{21})\\&=\bp^4  R\left((J)\a_{11}\a_{22}\right)+\bp^4 Z(\a_{22}J\a_{21})\\&=\partial_rR\bp^4\zeta_j\a_{j1}
\end{split}
\end{equation*}
Combining these two observations, we arrive at the following symmetric structure:
\begin{equation*}
\begin{split}
I_b&=\int_{\bT}R_0\partial_r\left(\hal\dfrac{C(t)^2}{{R}^2}-q\right)\a_{j1}\bp^4 \zeta_j  \a_{i1} \bp^4 \nu_i\,dz+l.o.t
\end{split}
\end{equation*}
Then under the Rayleigh-Taylor condition, the tangential energy estimates can be finished. Doing the divergence and curl estimates is somehow standard and combining with the tangential energy estimates, we can close the a priori estimates.

After we obtaining the a priori estimates, we construct approximate system to \eqref{eq:mhd}, which is asymptotically consistent with the a priori estimates for the original system. This is highly nontrivial. Recalling that, under the Rayleigh-Taylor sign condition, the a priori estimates relies heavily on the geometric transport-type structure of the nonlinear problem, which will lost during the linearization approximation. Hence, we apply the nonlinear $\kappa$-approximation developed in \cite{GW_16} and we can derive $\kappa$-independent a priori estimates. What now remains in the proof of the local well-posedness of \eqref{eq:mhd} is to constructing solutions to the nonlinear $\kappa$-approximate problem \eqref{sub1} and \eqref{sub2}. This solvability can be obtained by the viscosity vanishing method used in \cite[Section 5.1]{GW_16}. Consequently, the construction of solutions to the incompressible MHD equations \eqref{eq:mhd} under the Rayleigh-Taylor sign condition is completed.

The local well-posedness under a more general  ``stability'' assumption for the initial data, which provided that the Rayleigh-Taylor sign condition is satisfied at all those points of the initial interface where the non-collinearity condition fails can be obtained by following idea. The main difficulty is still to obtain the estimate of $I_b$ in \eqref{tp}. We split the boundary into two parts, one part contains the points that non-collinearity condition holds, then the Rayleigh-Taylor sign condition is satisfied at another part. As a consequence, $I_b$ is estimated by the combination of two parts. For one part, we use the non-collinearity condition to improve the boundary regularity and use $(H^{-1/2}, H^{1/2})$ dual estimate. For another part, we use the symmetric structure with the Rayleigh-Taylor sign condition. Thus, the estimate of $I_b$ is obtained and we can prove the local well-posedness.
\section{Preliminary}
{\subsection{Notation}\label{Notation}
Einstein's summation convention is used throughout the paper, and repeated
Latin indices $i,j,$ etc., are summed from 1 to 2.
We use $C$ to denote generic constants, which only depends on the domain $\Omega$ and the boundary $\Gamma$, and  use $f\ls g$ to denote $f\leq Cg$.  We use $P$ to denote a generic polynomial function of its arguments, and the polynomial coefficients are generic constants $C$. We use $D$ to denote the spatial derives: $\partial_r, \bp$.


\subsubsection{Sobolev spaces}
For integers $k\geq0$, we define the axially symmetic Sobolev space $H^k_{r,z}(\Omega)$ to be the completion of the functions in $C^{\infty}(\bar{\Omega})$ in the norm
\begin{equation*}
	\|u\|_k:=\left(\sum_{|\alpha|\leq k}2\pi\int_{\mathbb T}\int_{0}^{R_0}r\abs{D^{\alpha}u(r,z)}^2\,drdz\right)^{1/2}
\end{equation*}
for a multi-index $\alpha\in \mathbb{Z}_{+}^2$. For real numbers $s\geq0$, the Sobolev spaces $H^s_{r,z}(\Omega)$ are defined by interpolation.

\noindent
On the boundary $\Gamma$, for functions $w\in H^k(\Gamma)$, $k\geq0$, we set
\begin{equation*}
	|w|_k:=\left(\sum_{\beta\leq k}2\pi R_0\int_{\mathbb T} \abs{\partial_z^{\beta}w(R_0,z)}^2\,dz\right)^{1/2}
\end{equation*}
for a multi-index $\beta\in\mathbb{Z}_+$. The real number $s\geq0$ Sobolev space $H^s(\Gamma)$ is defined by interpolation. The negative-order Sobolev spaces $H^{-s}(\Gamma)$ are defined via duality: for real $s\geq0$, $H^{-s}(\Gamma):=[H^s(\Gamma)]'$.
\subsection{Product and commutator estimates}

We recall the following product and commutator estimates.
\begin{lemma}
It holds that

\noindent $(i)$ For $|\alpha|=k\geq 0$,
\begin{equation}
\label{co0}
\norm{D^{\alpha}(gh)}_0 \ls \norm{g}_{k}\norm{h}_{[\frac{k}{2}]+2}+\norm{g}_{[\frac{k}{2}]+2}\norm{h}_{k}.
\end{equation}
$(ii)$ For $|\alpha|=k\geq 1$, we define the commutator
\begin{equation*}
[D^{\alpha}, g]h =D^{\alpha}(gh)-gD^{\alpha} h.
\end{equation*}
Then we have
\begin{align}
\label{co1}
&\norm{[D^{\alpha}, g]h}_0\ls\norm{D g}_{k-1}\norm{h}_{[\frac{k-1}{2}]+2}+\norm{D g}_{[\frac{k-1}{2}]+2}\norm{h}_{k-1}.
\end{align}
$(iii)$ For $|\alpha|=k\geq 2$, we define the symmetric commutator
\begin{equation*}
\left[D^{\alpha}, g, h\right] = D^{\alpha}(gh)-D^{\alpha}g h-gD^{\alpha} h.
\end{equation*}
Then we have
\begin{equation}
\label{co2}
\norm{\left[D^{\alpha}, g, h\right]}_0\ls\norm{D g}_{k-2}\norm{D h}_{[\frac{k-2}{2}]+2}+ \norm{D g}_{[\frac{k-2}{2}]+2}\norm{D h}_{k-2} .
\end{equation}
\end{lemma}
\begin{proof}
The proof of these estimates is standard; we first use the Leibniz formula to expand these terms as sums of products and then control the $L^2_{r,z}$ norm of each product with the lower order derivative term in $L^\infty\subset H^2_{r,z} $ and the higher order derivative term in $L^2_{r,z}$. See for instance Lemma A.1 of \cite{Wang_15}.
\end{proof}

We will also use the following lemma.
\begin{lemma}
It holds that
\begin{equation}
\label{co123}
\abs{gh}_{1/2} \ls \abs{g}_{W^{1,\infty}}\abs{h}_{1/2}.
\end{equation}
 \end{lemma}
\begin{proof}
It is direct to check that $\abs{gh}_{s} \ls \abs{g}_{W^{1,\infty}}\abs{h}_{s}$ for $s=0,1$. Then the estimate \eqref{co123} follows by the interpolation.
\end{proof}
\subsection{Horizontal convolution-by-layers and commutation estimates}

As \cite{CS07,DS_10}, we will use the operation of horizontal convolution-by-layers which is defined as follows. Let $0\le \rho(z_\ast)\in C_0^{\infty}(\bR)$ be a standard mollifier such that $\text{spt}(\rho)=\overline{B(0,1)}$ and $\int_{\bR} \rho \,dz_{\ast}=1$, with corresponding dilated function $\rho_{\kappa}(z_\ast)=\frac{1}{\kappa}\rho(\frac{z_\ast}{\kappa}), \kappa>0$. We then define
\begin{equation}\label{lambdakdef}
\Lambda_{\kappa}g(r, z)=\int_{\bR}\rho_{\kappa}(z-z_{\ast})g(r, z_{\ast})\,dz_{\ast}.
\end{equation}
By standard properties of convolution, the following estimates hold:
\begin{align}
&\abs{\Lambda_{\kappa}h}_s\ls \abs{h}_s,\quad  s\ge 0, \label{test3}\\
&\abs{\bp\Lambda_{\kappa}h}_0\ls \dfrac{1}{\kappa^{1-s}}\abs{h}_s,\quad 0\le s\le 1.\label{loss}
\end{align}

The following commutator estimates play an important role in the boundary estimates.
\begin{lemma}\label{comm11}
For $\kappa>0$, we define the commutator
\begin{equation}
\left[\Lambda_{\kappa}, h\right]g\equiv \Lambda_{\kappa}(h g)-h\Lambda_{\kappa}g.
\end{equation}
Then we have
\begin{align}
&\abs{[\Lambda_{\kappa}, h]g}_0\ls  \abs{h}_{L^\infty}|g|_0,\label{es0-0}\\
&\abs{[\Lambda_{\kappa}, h]\bp g}_0\ls \abs{h}_{W^{1,\infty}}|g|_0,\label{es1-0}\\
&\abs{[\Lambda_{\kappa}, h]\bp g}_{1/2}\ls \abs{h}_{W^{1,\infty}}\abs{g}_{1/2}.\label{es1-1/2}
\end{align}
\end{lemma}
\subsection{Hardy-type inequality}
We recall the following Hardy inequality:
\begin{lemma}[A higher order Hardy-type inequality]
\label{hd1}
  Let $s\geq 1$ be a given integer, and suppose that $g\in H^s_{r,z}(\Omega)$
and $g(0,z)=0$, we have
  \begin{equation}
	 \norm{\dfrac{g}{r}}_{s-1}\leq C\norm{g}_s.
    \label{}
  \end{equation}
  \label{hardy}
\end{lemma}
Lemma \ref{hd1} can be proved by a similar approach used in \cite[Lemma 3.1]{Gu_2011}.
\subsection{Geometry Identities}
\begin{align}
\label{dJ}
&\partial J=\dfrac{\partial J}{\partial {\mathcal F}_{ij}}\partial {\mathcal F}_{ij} =  J{\mathcal{A}}_{ij}\partial {\mathcal F}_{ij},\\
&\partial {\mathcal{A}}_{ij}  = -{\mathcal{A}}_{i\ell}\partial {\mathcal F}_{m\ell}{\mathcal{A}}_{mj},
\label{partialF}
\end{align}
where $\partial$ can be $\partial_r$, $\partial_z$ and $\partial_t$ operators.

From the incompressible constraint, we have $\partial_t J= -J \dfrac{v^r}{R}$ for $J=\text{det} \mathcal{F}$, which means $J=\dfrac{r}{R}$.

Moreover, we have the Piola identity
\begin{equation}\label{polia}
	\partial_r\left(J\a_{1j}\right)+\partial_z\left(J\a_{2j}\right) =0.
\end{equation}
\section{Nonlinear approximate system}\label{lap}
In this section, we construct the nonlinear approximate system with the help of the horizontal convolution-by-layers. In the next two sections, we will derive a priori estimates for this system under the condition \eqref{stc} or condition \eqref{ts}.

First, we denote $\zeta=(R, Z), \nu=(v^r, v^z)$, and the matrix $\mathcal A^{\kappa}=\a(\zeta^\kappa)$ (and $J^{\kappa}$, etc.) with $\zeta^\kappa:=\zeta+\phi^{\kappa}$, $\phi^{\kappa}$ is the solution of the following elliptic equation:
\begin{equation}
\label{etadef}
\begin{cases}
-(\partial_r^2+\partial_z^2) \phi^{\kappa}=0 &\text{in }\Omega,\\
\phi^{\kappa}=\Lambda_{\kappa}^2\zeta-\zeta &\text{on }\Gamma,\\
\phi^{\kappa}(0,z,t)=0. &
\end{cases}
\end{equation}
Then we introduce the following two coupled nonlinear approximate sub-systems which are both defined in $\Omega$:
\begin{equation}
	\begin{cases}
		\partial_t \zeta= \nu+\fk\,\,&\text{in}\,\,\Omega,\\
\partial_t \nu+\nabla_{\ak} q-(b_0\cdot\nabla)^2\zeta=\left(\dfrac{(v^{\theta})^2}{R}-R(b_0\cdot\nabla \Theta)^2, 0\right)\,\,&\text{in}\,\,\Omega,\\
\Div_{\ak}\nu=0\,\,&\text{in}\,\,\Omega,\\
q=\dfrac{1}{2}\dfrac{C^2(t)}{{\Rk}^2}\,\,&\text{on}\,\,\Gamma,\\
(\zeta, \nu)|_{t=0}=(r, z, v_0^r, v_0^z).\,\,&\text{in}\,\,\Omega.
\end{cases}
	\label{sub1}
\end{equation}
and
\begin{equation}
	\begin{cases}
		\partial_t \Theta=\dfrac{v^{\theta}}{R}\,\,&\text{in}\,\,\Omega,	\\
	\partial_t v^{\theta}-(b_0\cdot\nabla)(Rb_0\cdot\nabla\Theta)=-\dfrac{v^{\theta}v^r}{R}
	+b_0\cdot\nabla Rb_0\cdot\nabla\Theta\,\,&\text{in}\,\,\Omega,\\(\Theta, v^{\theta})|_{t=0}=(\theta, v_0^{\theta}).
\end{cases}
	\label{sub2}
\end{equation}
where $\nabla_{\ak}=(\partial_R^{\ak}, \partial_Z^{\ak}), \partial_{\zeta_i}^{\ak}:=\ak_{ij}\partial_{a_j}, \Div_{\ak}g=\dfrac{1}{\Rk}\partial_{\zeta_i}^{\ak}(\Rk g_i).$

In the first equation of \eqref{sub1} we have introduced the modification term $\psi^{\kappa}=\psi^{\kappa}(\zeta,\nu)$ as the solution to the following elliptic equation
\begin{equation}
\label{etaaa}
\begin{cases}
-(\partial_r^2+\partial_z^2) \psi^{\kappa}=0&\text{in } \Omega,\\
\psi^{\kappa}= \int_0^z\int_0^{\tau} \mathbb{P}\left(\bp^2\zeta_{j}\a^{\kappa}_{j2}\partial_{z}{\Lambda_{\kappa}^2 v}-\bp^2{\Lambda_{\kappa}^2\zeta}_{j}\a^{\kappa}_{j2}\partial_z \nu\right)(r, \tau', t)\,d\tau' d\tau &\text{on }\Gamma,\\
\psi^\kappa(0,z,t)=0. &
\end{cases}
\end{equation}
where $\mathbb{P} f=f- \int_{\mathbb T}f$.
\section{A priori estimates with the generalized Rayleigh-Taylor sign condition \eqref{stc}}\label{ae1}

  In this Section, we derive a priori estimates for the approximate system \eqref{sub1} and \eqref{sub2} provided the condition \eqref{stc}. We take the time $T_{\kappa}>0$  sufficiently small so that for $t\in[0,T_{\kappa}]$,
\begin{align}
\label{ini2}
&-\partial_r\left(q (t)-\hal\dfrac{C(t)^2}{{\Rk}^2}\right)\geq \dfrac{\lambda}{2} \text{ on }\Gamma,\\
\label{inin3}&\abs{J^{\kappa}(t)-1}\leq \dfrac{1}{8} \text{ and } \abs{\a_{ij}^{\kappa}(t)-\delta_{ij}}\leq \dfrac{1}{8} \text{ in }\Omega.
\end{align}
We define the high order energy functional:
\begin{equation}
	\mathfrak{E}^{\kappa}(t)=\norm{(v^r, v^{\theta},  v^z, R, Z, b_0\cdot\nabla R, Rb_0\cdot\nabla \Theta, b_0\cdot\nabla Z)}_4^2
\end{equation}We will prove that $\mathfrak{E}^{\kappa}$ remains bounded on a time interval independent of $\kappa$, which is stated as the following theorem.

\begin{theorem} \label{th43}
There exists a time $T_1$ independent of $\kappa$ such that
\begin{equation}
\label{bound}
\sup_{[0,T_1]}\mathfrak{E}^{\kappa}(t)\leq 2M_0,
\end{equation}
where $M_0=P\left(\norm{\left(v_0^r, v_0^{\theta}, v_0^z\right)}_4^2+\norm{\left(b_0^r, b_0^{\theta}, b_0^z\right)}_4^2\right).$
\end{theorem}
\subsection{A priori estimates for system \eqref{sub1}.}\label{apest}
For system \eqref{sub1}, we have the following a priori estimates:
\begin{proposition} \label{th412}
For $t\in[0,T]$, it holds that:
	\begin{equation}
	\norm{\nu(t)}_4^2+\norm{\zeta(t)}_4^2+\norm{b_0\cdot\nabla\zeta(t)}_4^2+\abs{\bp^4\Lambda_{\kappa}\zeta_i \a^{\kappa}_{i1}(t)}_0^2\leq M_0+TCP\left(\sup_{t\in [0,T]}\mathfrak E^{\kappa}(t)\right).
\end{equation}
\end{proposition}
\subsubsection{Preliminary estimates of $\eta^{\kappa}$ and $\psi^{\kappa}$}

We begin our estimates with the boundary smoother $\zeta^{\kappa}$ defined by \eqref{etadef} and the modification term $\psi^{\kappa}$ defined by \eqref{etaaa}.

\begin{lemma}
\label{preest}
The following estimates hold:
\begin{align}
\label{tes1}\norm{\zeta^{\kappa}}_4&\ls \norm{\zeta}_4,\\
\label{tes2}\norm{b_0\cdot\nabla\zeta^{\kappa}}_4^2&\leq P(\norm{\zeta}_4,\norm{b_0}_4,\norm{b_0\cdot\nabla\zeta}_4),\\
\label{tes0}\norm{\partial_t\zeta^{\kappa} }_4&\leq P(\norm{\zeta}_4, \norm{\nu}_4),\\
\label{fest1}\norm{\fk}_{4}&\leq P(\norm{\zeta}_4,\norm{\nu}_3),\\
\label{fest22}\norm{b_0\cdot\nabla \fk}_4&\leq P(\norm{\zeta}_4,\norm{\nu}_4,\norm{b_0\cdot\nabla\zeta}_4),\\
\label{fest2}\norm{\partial_t \fk}_4&\leq P(\norm{\zeta}_4,\norm{\nu}_4,\norm{\partial_t \nu}_3).
\end{align}
\end{lemma}
\begin{proof}
First, the standard elliptic regularity theory on the problem \eqref{etadef}, the trace theorem and the estimate \eqref{test3} yield
\begin{equation*}
\norm{\zeta^{\kappa}}_4\ls \norm{\zeta}_4+\abs{\Lambda_{\kappa}^2\zeta-\zeta}_{7/2}\ls\norm{\zeta}_4+\abs{\zeta}_{7/2}\ls \norm{\zeta}_4,
\end{equation*}
which implies \eqref{tes1}. To prove \eqref{tes2}, we apply $b_0\cdot\nabla=b_0^r\partial_r+b_0^z\partial_z$ to \eqref{etadef} to find that, since $b_0^r=0$ on $\Gamma$,
\begin{equation*}
\begin{cases}
-(\partial_r^2+\partial_z^2)(b_0\cdot\nabla\phi^{\kappa})=\left[b_0\cdot\nabla, \partial_r^2+\partial_z^2\right]\phi^{\kappa} &\text{in } \Omega,\\
b_0\cdot\nabla\phi^{\kappa}=b_0\cdot\nabla\Lambda_{\kappa}^2\zeta-b_0\cdot\nabla\zeta &\text{on } \Gamma,\\
b_0\cdot\nabla\phi^{\kappa}(0,z,t)=0. &
\end{cases}
\end{equation*}
We then have, since $H^4_{r,z}$ is a multiplicative algebra and by \eqref{tes1},
\begin{equation*}
\begin{split}
\norm{b_0\cdot\nabla\zeta^{\kappa}}_4&\ls \norm{b_0\cdot\nabla\zeta}_4+\norm{\left[b_0\cdot\nabla, \partial_r^2+\partial_z^2\right]\phi^{\kappa}}_2
+\abs{b_0\cdot\nabla\Lambda_{\kappa}^2\zeta-b_0\cdot\nabla\zeta}_{7/2}
\\&\ls \norm{b_0\cdot\nabla\zeta}_4+\norm{b_0}_4\norm{ \phi^\kappa}_4
+\abs{\Lambda_{\kappa}^2(b_0\cdot\nabla\zeta)}_{7/2}+\abs{[b_0\cdot\nabla, \Lambda_{\kappa}^2]\zeta}_{7/2}+\abs{b_0\cdot\nabla\zeta}_{7/2}
\\&\ls \norm{b_0\cdot\nabla\zeta}_4+\norm{b_0}_4\norm{\phi^\kappa}_4+\abs{b_0\cdot\nabla\zeta}_{7/2}
+\abs{b_0}_{7/2}\abs{\zeta}_{7/2} \\
&\ls \norm{b_0\cdot\nabla\zeta}_4+\norm{b_0}_4\norm{\zeta}_4.
\end{split}
\end{equation*}
Here we have used the estimates \eqref{es1-1/2} to estimate
\begin{equation*}
\begin{split}
 \abs{[b_0\cdot\nabla, \Lambda_{\kappa}^2]\zeta}_{7/2}&\le \abs{[\Lambda_{\kappa}^2,b_0^z]\pa_z\zeta}_{1/2}+\abs{[\Lambda_{\kappa}^2,b_0^z]\pa_z \bp^3\zeta}_{1/2}
 +\abs{\left[\bp^3,[\Lambda_{\kappa}^2,b_{0}^z]\pa_z \right]\zeta}_{1/2}
\\&\ls   \abs{b_0}_{W^{1,\infty}}\abs{\zeta}_{7/2}
+\abs{b_0}_{7/2}\abs{\zeta}_{7/2}\ls \abs{b_0}_{7/2}\abs{\zeta}_{7/2}.
\end{split}
\end{equation*}
This proves \eqref{tes2}.

We now turn to prove \eqref{fest1}. By the boundary condition in \eqref{etaaa} and the elliptic theory, we obtain, using the identity \eqref{partialF}, the a priori assumption \eqref{inin3} and the estimates \eqref{tes1},
\begin{equation*}
\begin{split}
\abs{\fk}_{7/2}&\ls\abs{\bp^2\zeta_{j}\a^{\kappa}_{j2}\partial_z{\Lambda_{\kappa}^2 \nu}-\bp^2{\Lambda_{\kappa}^2\zeta}_{j}\a^{\kappa}_{j2}\partial_z \nu}_{3/2}
\ls\norm{\bp^2\zeta_{j}\a^{\kappa}_{j2}\partial_z{\Lambda_{\kappa}^2 \nu}-\bp^2{\Lambda_{\kappa}^2\zeta}_{j}\a^{\kappa}_{j2}\partial_{z} \nu}_{2}
\\&\ls\norm{ \zeta}_4\norm{\a^{\kappa}}_{2}\norm{\nu}_{3}
\leq P(\norm{\zeta}_4,\norm{\nu}_3).
\end{split}
\end{equation*}
This proves \eqref{fest1} by using further the elliptic theory and the trace theorem.

Finally, the estimate \eqref{tes0} can be obtained similarly as \eqref{tes1} by applying $\partial_t$ to \eqref{etadef} and then using the equation $\partial_t\zeta =\nu+\fk$ and the estimate \eqref{fest1}.
The estimates \eqref{fest22} and \eqref{fest2} could be achieved similarly as \eqref{tes2} and \eqref{tes0} by applying $b_0\cdot\nabla$ and $\partial_t$ to \eqref{etaaa} and using the estimates \eqref{tes1}--\eqref{fest1}. This concludes the lemma.
\end{proof}

\subsubsection{Transport estimates of  $\zeta$}

The transport estimate of $\zeta$ is recorded as follows.
\begin{proposition}
For $t\in [0,T]$ with $T\le T_\kappa$, it holds that
\begin{equation}\label{etaest}
\norm{\zeta(t)}_4^2\leq M_0+TP\left(\sup_{t\in[0,T]} \mathfrak{E}^{\kappa}(t)\right).
\end{equation}
\end{proposition}
\begin{proof}
It follows by using  $\partial_t\zeta=\nu+\fk$ and the estimate \eqref{fest1}.
\end{proof}

\subsubsection{Pressure estimates}\label{pressure1}
\begin{proposition}
The following estimate holds:
	\begin{equation}
\label{pressure}
	\norm{q}_4^2+\norm{\partial_tq }_3^2\leq CP\left(\sup_{t\in [0,T]}\mathfrak E(t)\right).
\end{equation}
\label{pr}
\end{proposition}
\begin{proof}
Taking $J^{\kappa}\Div_{\ak}$ on the second equation of \eqref{sub1} to get:
\begin{equation}
	\begin{cases}
		\dfrac{1}{\Rk}\partial_{a_i}(\Rk E^{\kappa}_{ij}\partial_{a_j} q)=G_1+b_0\cdot\nabla G_2 \\
	q|_{\Gamma}=\dfrac{C^2(t)}{{\Rk}^2}
\end{cases}
	\label{ell}
\end{equation}
where
\begin{equation*}
	\begin{split}
		&E^{\kappa}_{ij}:=J^{\kappa}\ak_{\ell i}\ak_{\ell j},\\
		&G_1:=\Jk\left[\Div_{\ak},\partial_t\right]\nu+\left(\dfrac{\Jk}{\Rk}+\Jk\pa^{\ak}_R\right)\left(\dfrac{(v^{\theta})^2}{R}-R(b_0\cdot\nabla\Theta)^2\right)+\left[\Jk\Div_{\ak}, b_0\cdot\nabla\right]b_0\cdot\nabla \zeta\\
	&G_2:=\Jk\Div_{\ak}(b_0\cdot\nabla\zeta)
\end{split}
\end{equation*}
Note that  by \eqref{inin3} the matrix $\Ek$ is symmetric and positive.

We denote $\hat h(r,z,t)$ as the harmonic extension of $\frac{C^2(t)}{{\Rk}^2}$:
\begin{equation*}
\begin{cases}
(\partial_r^2+\dfrac{1}{r}\partial_r+\partial_z^2)\hat h =0\,\,&\text{in } \Omega,\\
\hat h= \frac{C^2(t)}{{\Rk}^2}\,\,&\text{on } \Gamma,
\end{cases}
\end{equation*}
and by the Trace theorem, we have
\begin{equation}
\label{hq}
	\norm{\hat h}_4^2\ls \abs{\dfrac{C^2(t)}{{\Rk}^2(R_0,z)}}_{3.5}^2\leq CP\left(\sup_{t\in [0,T]}\mathfrak{E}^{\kappa}(t)\right).
\end{equation}
And then $\hat q = q-\hat h$ satisfying the following elliptic equation with zero Dirichlet boundary condition:
\begin{equation}
	\dfrac{1}{\Rk}\partial_{a_i}(\Rk E^{\kappa}_{ij}\partial_{a_j}\hat q)=G_1+b_0\cdot\nabla G_2+	\dfrac{1}{\Rk}\partial_{a_i}(\Rk E^{\kappa}_{ij}\partial_{a_j} \hat h)
	\label{pressure2}
\end{equation}
Timing $\frac{\Rk}{r}\hat q$ on the equation \eqref{pressure2}, integrating on $\Omega$ and using integration-by-parts, we have
\begin{equation*}
	\int_{\Omega}\dfrac{\Rk}{r}E^{\kappa}_{ij}\partial_{a_j} \hat q\partial_{a_{i}} \hat q\,dx=\int_{\Omega}\dfrac{\Rk}{r}G_1\hat q\,dx+\int_{\Omega} \dfrac{\Rk}{r}G_2 b_0\cdot\nabla\hat q\,dx+\int_{\Omega}\dfrac{\Rk}{r}E^{\kappa}_{ij}\partial_{a_j} \hat h\partial_{a_i}\hat q\,dx	
\end{equation*}
Thus, with a priori assumption \eqref{inin3} and Poincare's inequality, we arrive at
\begin{equation}
	\norm{D \hat q}_0^2\ls \norm{G_1}_0^2+\norm{G_2}_0^2+\norm{E^{\kappa}_{ij}\partial_{a_j} \hat h}_0^2\leq CP\left(\sup_{t\in [0,T]}\mathfrak{E}^{\kappa}(t)\right).
	\end{equation}
and hence with \eqref{hq}, we have
\begin{equation}
\label{dfddd}
\norm{q}_1^2\ls P\left(\sup_{t\in [0,T]}\mathfrak{E}^{\kappa}(t)\right).
\end{equation}
Next, applying $\partial_z^{k}$, $k=1,2,3$ to the equation \eqref{pressure2} leads to
\begin{equation*}
\begin{split}
	\dfrac{1}{\Rk}\partial_{a_i}(\Rk E^{\kappa}_{ij}\partial_{a_j} \partial_z^k\hat q)=&\partial_z^kG_1+b_0\cdot\nabla \partial_z^kG_2+	\dfrac{1}{\Rk}\partial_{a_i}(\Rk E^{\kappa}_{ij}\partial_{a_j} \partial_z^k\hat h)\\&+\dfrac{1}{\Rk}\partial_{a_i}\left(\left[\partial_z^k,\Rk E^{\kappa}_{ij}\partial_{a_j}\right](\hat h-\hat q)\right)+\left[\partial_z^k, b_0\cdot\nabla\right]G_2.
\end{split}
\end{equation*}
Thus, similarly, we obtain
\begin{equation*}
\begin{split}
\norm{\partial_z^k \hat q}_1^2\ls& \norm{\partial_z^k G_1}_0^2+\norm{\partial_z^k G_2}_0^2+\norm{E^{\kappa}_{ij}\partial_{a_j}\partial_z^k \hat h}_0^2+\norm{\left[\partial_z^k,\Rk E^{\kappa}_{ij}\partial_{a_j}\right](\hat h-\hat q)}_0^2\\&+\norm{[\partial_z^k, b_0\cdot\nabla]G_2}_0^2,
\end{split}
\end{equation*}
and then
\begin{equation*}
	\norm{\partial_z^k \hat q}_1^2\leq C\left(P\left(\sup_{t\in [0,T]}\mathfrak{E}^{\kappa}(t)\right)+\norm{\partial_z^{k-1}\hat q}_1^2\right).
\end{equation*}
Combining with \eqref{hq} again, we have
\begin{equation}
\label{prs}
\norm{\partial_z^k q}_1^2\leq C\left(P\left(\sup_{t\in [0,T]}\mathfrak{E}^{\kappa}(t)\right)+\norm{\partial_z^{k-1} q}_1^2\right).
\end{equation}
In order to obtain other high order derivatives of $q$, we denote $\mathfrak g=\dfrac{\Ek_{1j}\partial_{a_j}q}{\Rk}$ and rewrite the first equation of \eqref{ell} as
\begin{equation}
	\label{pressure3}
	\dfrac{1}{\Rk}\partial_r({\Rk}^2 \mathfrak g)={\Rk}\partial_r\mathfrak g+2\partial_r{\Rk}\mathfrak g=\mathfrak G
\end{equation}
where
\begin{equation*}
\mathfrak G:=\left(G_1+b_0\cdot\nabla G_2+\dfrac{1}{{\Rk}}\partial_z({\Rk}\Ek_{2j}\partial_{a_j} q)\right).
\end{equation*}
Then we obtain
\begin{equation}
	\norm{{\Rk}\partial_r\mathfrak g+2\partial_r{\Rk}\mathfrak g}_0^2\leq  \norm{\mathfrak G}_0^2\leq CP\left(\sup_{t\in [0,T]}\mathfrak{E}^{\kappa}(t)\right).
	\label{jhieoe}
\end{equation}
With integration-by-parts and a priori assumption \eqref{inin3},  we have
\begin{equation}
\begin{split}
	&\int_{\Omega}\left({\Rk}\partial_r\mathfrak g+2\partial_r{\Rk}\mathfrak g\right)^2\,dx\\=	&\norm{{\Rk} \partial_r\mathfrak g}_0^2+4\norm{\partial_r{\Rk}\mathfrak g}_0^2+\int_{\Omega}4{\Rk}\partial_r\mathfrak g\partial_r{\Rk}\mathfrak g\,dx\\=&\norm{{\Rk} \partial_r\mathfrak g}_0^2+2\norm{\partial_r{\Rk}\mathfrak g}_0^2-2\int_{\mathbb T}\int_0^{R_0}\partial_r(r\partial_r{\Rk}){\Rk}\abs{\mathfrak g}^2\,drdz+2\int_{\mathbb T}R_0{\Rk}(R_0,z)\abs{\mathfrak g}^2(R_0, z)\,dz\\\geq &\norm{{\Rk} \partial_r\mathfrak g}_0^2+2\norm{\partial_r{\Rk}\mathfrak g}_0^2-CT\sup_{t\in [0,T] }\abs{\partial_r(r\partial_r \partial_t {\Rk})}_{L^{\infty}}\norm{\mathfrak g}_0^2\\\geq & \norm{{\Rk} \partial_r\mathfrak g}_0^2+\norm{\mathfrak g}_0^2
\end{split}
	\label{labelss}
\end{equation}
by taking $T$ sufficiently small (only depend on $M$).
Thus,  we arrive at
\begin{equation}
	\norm{{\Rk} \partial_r\mathfrak g}_0^2+\norm{\mathfrak g}_0^2\leq CP\left(\sup_{t\in [0,T]}\mathfrak{E}^{\kappa}(t)\right)
	\label{ef}
\end{equation}
and as a consequence, we have
	\begin{equation}
		\norm{\partial_r (\Ek_{1j}\partial_{a_j} q)}_0^2
\leq \norm{\partial_r{\Rk}\mathfrak g}_0^2+\norm{{\Rk}\partial_r\mathfrak g}_0^2\leq CP\left(\sup_{t\in [0,T]}\mathfrak{E}^{\kappa}(t)\right).
		\label{jifiefjoq}
	\end{equation}
Then by using \eqref{prs} and a priori assumption \eqref{inin3} again, we have
\begin{equation}
\label{riep}
\norm{\partial^2_r q}_0^2\leq \norm{\dfrac{1}{\Ek_{11}}\partial_r (\Ek_{1j}\partial_{a_j} q)}_0^2+\norm{\dfrac{1}{\Ek_{11}}\partial_r \Ek_{11}\partial_r q}_0^2+\norm{\dfrac{1}{\Ek_{11}}\partial_r (\Ek_{12}\partial_z q)}_0^2 \leq CP\left(\sup_{t\in [0,T]}\mathfrak{E}^{\kappa}(t)\right).
\end{equation}
Next, by acting $\partial_z, \partial_z^2$ on the equation \eqref{pressure3}, we have
\begin{equation*}
\begin{split}
{\Rk}\partial_r\partial_z\mathfrak g+2\partial_r{\Rk}\partial_z\mathfrak g=&-\partial_z{\Rk}\partial_r\mathfrak g-2\partial_z\partial_r{\Rk}\mathfrak g+\partial_z \mathfrak G\\
{\Rk}\partial_r\partial_z^2\mathfrak g+2\partial_r{\Rk}\partial_z^2\mathfrak g=&-\partial_z{\Rk}\partial_r\partial_z\mathfrak g-2\partial_r\partial_z{\Rk}\partial_z\mathfrak g+\partial_z\left(-\partial_z{\Rk}\partial_r\mathfrak g-2\partial_z\partial_r{\Rk}\mathfrak g+\partial_z \mathfrak G\right)
\end{split}
\end{equation*}
Then by a similar approach from \eqref{jhieoe} to \eqref{riep}, for $k=1,2,$ we can obtain
\begin{equation*}
	\norm{{\Rk} \partial_r\partial_z^k\mathfrak g}_0^2+\norm{\partial_z^k\mathfrak g}_0^2\leq CP\left(\sup_{t\in [0,T]}\mathfrak{E}^{\kappa}(t)\right)
\end{equation*}
and hence
\begin{equation*}
\norm{\partial_r^2\partial_z^2 q}_0^2\leq CP\left(\sup_{t\in [0,T]}\mathfrak{E}^{\kappa}(t)\right).
\end{equation*}
Finally, we act $\partial_r, \partial_r^2$ on the equation \eqref{pressure3} to obtain for $k=1,2,$
\begin{equation}
\norm{{\Rk} \partial_r^{k+1}\mathfrak g}_0^2+\norm{\partial_r^k\mathfrak g}_0^2\leq CP\left(\sup_{t\in [0,T]}\mathfrak{E}^{\kappa}(t)\right)
	\label{ipre}
\end{equation}
and hence
\begin{equation}
\norm{\partial_r^4 q}_0^2+\norm{\partial_r^3\partial_z q}_0^2\leq CP\left(\sup_{t\in [0,T]}\mathfrak{E}^{\kappa}(t)\right).
\end{equation}
Combining with \eqref{prs}, we obtain
\begin{equation}
\label{3q}
\norm{q}_4^2\leq P(\mathfrak E^{\kappa}).
\end{equation}

We now estimate $\partial_t q$. Applying $\partial_t$ to the equation \eqref{ell} and by arguing similarly as for \eqref{3q}, we can obtain
\begin{equation}\label{3qt}
\norm{ \dt q}^2_3\leq P\left( \mathfrak{E}^{\kappa} \right).
\end{equation}
Here we have used the estimates \eqref{tes1}--\eqref{fest2} and noted that by using the second equation in \eqref{sub1} and the estimates \eqref{3q}:
\begin{equation}
\label{vt}
\norm{\partial_t\nu}_3^2=\norm{-\nabak q+(b_0\cdot \nabla)^2 \zeta+\left(\dfrac{(v^{\theta})^2}{R}-R(b_0\cdot\nabla \Theta)^2, 0\right)}_3^2 \leq P\left( \mathfrak{E}^{\kappa} \right).
\end{equation}
Consequently, the estimates \eqref{3q} and \eqref{3qt} give \eqref{pressure}.

\end{proof}
\subsubsection{Tangential estimates for $\nu=(v^r, v^z)$}
We start with the basic $L^2$ energy estimates.
\begin{proposition}\label{basic}
For $t\in [0,T]$ with $T\le T_\kappa$, it holds that
\begin{equation}\label{00estimate}
\norm{\nu(t)}_0^2+\norm{(b_0\cdot\nabla\zeta)(t)}_0^2\leq M_0+TP\left(\sup_{t\in[0,T]}\mathfrak{E}^{\kappa}(t)\right).
\end{equation}
\end{proposition}
\begin{proof}
Taking the $L^2(\Omega)$ inner product of the second equation in \eqref{sub1} with $\nu$ yields
\begin{equation}\label{hhl1}
 \dfrac{1}{2}\dfrac{d}{dt}\int_{\Omega}\abs{\nu}^2 +\int_{\Omega} \nak q\cdot \nu -\int_{\Omega}(b_0\cdot\nabla)^2\zeta \cdot \nu =\int_{\Omega}\left(\dfrac{(v^{\theta})^2}{R}-R(b_0\cdot\nabla \Theta)^2\right) v^r.
\end{equation}
Using the pressure estimates \eqref{pressure}, we have
\begin{equation}
-\int_{\Omega} \nak q\cdot \nu\le\norm{D \a^\kappa}_{L^{\infty}}\norm{\nu}_0\norm{q}_1\le P\left(\sup_{t\in[0,T]}\mathfrak{E}^{\kappa}(t)\right).
\end{equation}
By Hardy's inequality, we have
\begin{equation}
\int_{\Omega}\left(\dfrac{(v^{\theta})^2}{R}-R(b_0\cdot\nabla \Theta)^2\right) v^r \le P\left(\sup_{t\in[0,T]}\mathfrak{E}^{\kappa}(t)\right).
\end{equation}
Since $b_0$ satisfies \eqref{bcond}, by the integration by parts and using $\dt \zeta=\nu+\fk$, we obtain
\begin{equation}\label{hhl3}
\begin{split}
-\int_{\Omega}(b_0\cdot\nabla)^2\zeta\cdot v&=\int_{\Omega}b_0\cdot\nabla\zeta_ib_0\cdot\nabla \nu_i\\&=\int_{\Omega}b_0\cdot\nabla\zeta_ib_0\cdot\nabla \partial_t\zeta_i\,dx-\int_{\Omega}b_0\cdot\nabla\zeta_ib_0\cdot\nabla \psi^\kappa_i
\\&=\dfrac{1}{2}\dfrac{d}{dt}\int_{\Omega}\abs{b_0\cdot\nabla\zeta}^2\,dx-\int_{\Omega}b_0\cdot\nabla\eta_ib_0\cdot\nabla \psi^\kappa_i
\end{split}
\end{equation}
Then \eqref{hhl1}--\eqref{hhl3} implies, using the estimates \eqref{fest1},
\begin{equation}
\dfrac{d}{dt}\int_{\Omega}\abs{\nu}^2+\abs{b_0\cdot\nabla\zeta}^2 \leq P\left(\sup_{t\in[0,T]}\mathfrak E^{\kappa}(t)\right).
\end{equation}
Integrating directly in time of the above yields \eqref{00estimate}.
\end{proof}
In order to perform higher order tangential energy estimates, one needs to compute the equations satisfied by $(\bp^4 \nu, \bp^4 q, \bp^4 \zeta)$, which requires to commutate $\bp^4$ with each term of $\pa^{\ak}_{\zeta_i}$. It is thus useful to establish the following general expressions and estimates for commutators.
 we have
\begin{equation}
	\bp^4 (\pa^{\ak}_{\zeta_i}g) =  \pa^{\ak}_{\zeta_i} \bp^4 g + \bp^4 \ak_{ij} \pa_{a_j} g+\left[\bp^4, \ak_{ij} ,\pa_{a_j} g\right].
\end{equation}
By the identity \eqref{partialF}, we have that
\begin{equation}
\begin{split}
&\bp^4 (\ak_{ij} \pa_{a_j} g)=-\bp^3({\ak}_{i\ell}\bp\pa_{a_{\ell}}  \zeta_m  \ak_{mj})\pa_{a_j} g
\\&\quad=-\ak_{i\ell}\pa_{a_\ell} \bp^4\zeta^m\ak_{mj}\pa_{a_j} g
-\left[\bp^3, \ak_{i\ell}\ak_{mj} \right]\bp \pa_{a_\ell}  \zeta_m \pa_{a_j} g
\\&\quad=-\pa^{\ak}_{\zeta_i}( \bp^4\zeta\cdot\nabla_{\ak} g)+\bp^4\zeta\cdot\nabla_{\ak}( \pa^{\ak}_{\zeta_i} g)
	-\left[\bp^3, \ak_{i\ell}\ak_{mj}\right]\bp\pa_{a_\ell} \zeta_m\pa_{a_j} g.\end{split}
\end{equation}
It then holds that
\begin{equation}\label{commf}
	\bp^4 (\pa^{\ak}_{\zeta_i}g) =  \pa^{\ak}_{\zeta_i}\left(\bp^4 g- \bp^4\zeta\cdot\nabla_{\ak} g\right)+ \mathcal{C}_i(g).
\end{equation}
where the commutator $\mathcal{C}_i(g)$ is given by
\begin{equation}
	\mathcal{C}_i(g)=\left[\bp^4, {\ak_{ij}},\pa_{a_j} g\right]-\bp^4\zeta\cdot\nabla_{\ak}( \pa^{\ak}_{\zeta_i}g)
	+\left[\bp^3, \ak_{i\ell}\ak_{mj}\right]\bp\pa_{a_\ell} \zeta_m\pa_{a_j} g
\end{equation}
It was first observed by Alinhac \cite{Alinhac} that the highest order term of $\zeta$ will be canceled when one uses the good unknown $\bp^4 g- \bp^4\zeta\cdot\nabla_{\ak} g$, which allows one to perform high order energy estimates.

The following lemma deals with the estimates of the commutator $\mathcal C_i(g)$.
\begin{lemma}
The following estimate holds:
\begin{equation}\label{comest}
	\norm{\mathcal C_i (g)}_0\leq  P\left(\norm{({\Rk},\Zk)}_4\right) \norm{g}_4.
\end{equation}
\end{lemma}
\begin{proof}
First, by the commutator estimates \eqref{co2}, we have
\begin{equation}\label{Calpha1}
	\norm{[\bp^4, \ak_{ij}, \partial_{a_j} g]}_0\ls\norm{\ak}_{3}\norm{D g}_{3}\leq P\left(\norm{({\Rk}, \Zk)}_4\right) \norm{g}_4.
\end{equation}
Next, we get
\begin{equation}
	\norm{\bp^4\zeta\cdot\nabla_{\ak}( \pa^{\ak}_{\zeta_i} g)}_0\le \norm{\bp^4\zeta}_0\norm{\nabla_{\ak}( \pa^{\ak}_{\zeta_i} g)}_{L^{\infty} } \leq P\left(\norm{({\Rk}, \Zk)}_4\right) \norm{g}_4.
\end{equation}
Finally, by the commutator estimates \eqref{co1}, we obtain
\begin{equation}\label{Calpha3}
	\norm{\left[\bp^3, \ak_{i\ell}\ak_{mj}\right]\bp\pa_{a_\ell}\zeta^m\pa_{a_j} g}_0\le \norm{\left[\bp^3, \ak_{i\ell}\ak_{mj}\right]\bp\pa_{a_\ell} \zeta_m}_0\norm{Dg}_{L^{\infty} }
	\leq P\left(\norm{({\Rk},\Zk)}_4\right) \norm{g}_3.
\end{equation}

Consequently, the estimate \eqref{comest} follows by collecting \eqref{Calpha1}--\eqref{Calpha3}.
\end{proof}

We now introduce the good unknowns
\begin{equation}
\label{gun}
\mathcal{V}=\bp^4 \nu- \bp^4\zeta\cdot\nabla_{\ak} \nu,\quad \mathcal{Q}=\bp^4 q- \bp^4\zeta\cdot\nabla_{\ak} q.
\end{equation}

With the condition \eqref{qbdd}, we have \begin{equation}
	\mathcal{Q}=\dfrac{1}{2}\bp^4\left(\dfrac{C^2(t)}{{\Rk}^2}\right)-\bp^4\zeta\cdot\nabla_{\ak} q\quad\text{on } \Gamma.
	\label{qbdc}
\end{equation}
Applying $\bp^4$ to the second equation of \eqref{sub1}, by \eqref{commf}, one gets
\begin{equation}\label{eqValpha}
	\begin{split}
		\partial_t\mathcal{V} &+ \nabla_{\ak}\mathcal{Q}-(b_0\cdot\nabla)\left(\bp^4(b_0\cdot\nabla \zeta)\right)\\&=F:=  \dt\left(\bp^4\zeta\cdot\nabla_{\ak} \nu\right) - \mathcal{C}_i(q) +\left[\bp^4, b_0\cdot\nabla\right]b_0\cdot\nabla \zeta+\bp^4\left(\dfrac{(v^{\theta})^2}{R}-{R}(b_0\cdot\nabla\Theta)^2\right),
\end{split}
\end{equation}
and
\begin{equation} \label{divValpha}
	\Div_{\ak}\mathcal{V}=g_3:=- \mathcal{C}_i(v^r)-\mathcal{C}_i(v^z)-\left(\bp^4(\dfrac{v^r}{{\Rk}})-\dfrac{\mathcal{V}_1}{{\Rk}}\right) .
\end{equation}

We shall now derive the $\bp^4$-energy estimates and have the following proposition
\begin{proposition}\label{te}
For $t\in [0,T]$, it holds that
\begin{equation}
\label{teee}
\norm{\partial_z^4 \nu}_0^2+\norm{\partial_z^4(b_0\cdot\nabla \zeta)(t)}_0^2+\abs{\bp^4\Lambda_{\kappa}\zeta_i \a^{\kappa}_{i1}(t)}_0^2\leq M_0+CTP\left(\sup_{t\in[0,T]}\mathfrak{E}^{\kappa}(t)\right).
\end{equation}
\end{proposition}
\begin{proof}
Taking the $L^2(\Omega)$ inner product of \eqref{eqValpha} with $\mathcal{V}$ yields
\begin{equation}
	\dfrac{1}{2}\dfrac{d}{dt}\int_{\Omega}|\mathcal{V}|^2\,dx+\int_{\Omega}\pa^\a_{\zeta_i}\mathcal Q \mathcal{V}_i\,dx+\int_{\Omega}\bp^4(b_0\cdot\nabla \zeta_i)b_0\cdot\nabla\mathcal{V}_i\,dx=\int_{\Omega}F\cdot\mathcal{V}\,dx.
	\label{fd}
\end{equation}
Firstly, the right hand side of \eqref{fd} can be bounded by
\begin{equation}
\label{ggg0}
\begin{split}
\quad\int_{\Omega}F\cdot\mathcal{V}\,dx &\leq \bigg(\norm{\dt\left(\bp^4\zeta\cdot\nabla_{\ak}\nu\right)}_0 - \norm{\mathcal{C}_i(q)}_0+\norm{\left[\bp^4, b_0\cdot\nabla\right]b_0\cdot\nabla\zeta}_0\\&\quad+\norm{\bp^4\left(\dfrac{(v^{\theta})^2}{R}-{R}(b_0\cdot\nabla\Theta)^2\bigg)}_0\right)\norm{\mathcal V}_0\\
&\leq CP\left(\sup_{t\in [0,T]}\mathfrak{E}^{\kappa}(t)\right).
\end{split}
\end{equation}
Here we use \eqref{comest} and we estimate
\begin{equation}
	\begin{split}
	\norm{\bp^4\left(\dfrac{(v^{\theta})^2}{{R}}\right)}_0&\leq C\left(\norm{{R} \bp^4(\dfrac{ v^{\theta}}{R})}_0\norm{\dfrac{ v^{\theta}}{{R}}}_{L^{\infty}}+\norm{ v^{\theta}}_4\norm{\dfrac{ v^{\theta}}{{R}}}_{L^{\infty}}\right)
	\\&\leq C\left(\norm{\bp^4 v^{\theta}-\left[\bp^4, \dfrac{v^{\theta}}{{R}}\right]{R}}_0 \norm{\dfrac{v^{\theta}}{{R}}}_{L^{\infty}}+\norm{v^{\theta}}_4\norm{\dfrac{ v^{\theta}}{{R}}}_{L^{\infty}}\right) \\&\leq CP\left(\sup_{t\in[0,T]}\mathfrak{E}^{\kappa}(t)\right)
\end{split}
\end{equation}
by using Hardy's inequality \eqref{hardy} and $\norm{\bp^4({R}(b_0\cdot\nabla\Theta)^2)}_0$ can also be bounded by $CP\left(\sup\limits_{t\in[0,T]}\mathfrak{E}^{\kappa}(t)\right)$.

Next, by the definition of $\mathcal V$ and recalling that $\nu=\partial_t\zeta-\fk$, we have
\begin{equation}
\label{gg1}
	\begin{split}
		&\int_{\Omega}\bp^4(b_0\cdot\nabla \zeta_i) b_0\cdot\nabla\mathcal{V}_i\,dx\\=&\int_{\Omega}\bp^4(b_0\cdot\nabla \zeta_i)b_0\cdot\nabla(\bp^4\nu_i)\,dx-\int_{\Omega}\bp^4(b_0\cdot\nabla \zeta_i)b_0\cdot\nabla(\bp^4\zeta\cdot\nabla_{\ak}\nu_i)\,dx\\=&\dfrac{1}{2}\dfrac{d}{dt}\int_{\Omega}|\bp^4(b_0\cdot\nabla \zeta)|^2\,dx+\int_{\Omega}\bp^4(b_0\cdot\nabla \zeta_i)\left[b_0\cdot\nabla,\bp^4\right]\nu_i\,dx\\&-\int_{\Omega}\bp^4(b_0\cdot\nabla \zeta_i)b_0\cdot\nabla(\bp^4\zeta\cdot\nabla_{\ak}\nu_i)\,dx-\int_{\Omega}\bp^4(b_0\cdot\nabla\zeta_i)   \left(\bp^4(b_0\cdot\nabla  \psi^\kappa_i)\right)\,dx\\\geq&\dfrac{1}{2}\dfrac{d}{dt}\left(2\pi\int_{\mathbb T}\int_{0}^{R_0}r|\bp^4(b_0\cdot\nabla \zeta)|^2\,drdz\right)-CP\left(\sup_{t\in [0,T]}\mathfrak{E}^{\kappa}(t)\right).
	\end{split}
\end{equation}
By integration-by-parts and \eqref{divValpha}, we have
\begin{equation}
	\begin{split}
		\int_{\Omega}\pa^{\ak}_{\zeta_i}\mathcal Q \mathcal{V}_i\,dx=&-\int_{\Omega}\dfrac{1}{{\Rk}}\mathcal{Q}\left(\pa^{\ak}_{\zeta_i}({\Rk}\mathcal{V}_i)\right)\,dx-\int_{\Omega}\pa_{a_j}\left(\dfrac{r}{{\Rk}}\ak_{ij}\right)\mathcal Q\dfrac{{\Rk}}{r}\mathcal{V}_i\,dx\\&+\underbrace{2\pi\int_{\mathbb T}R_0\mathcal Q\left(\mathcal{V}_i\ak_{i1}\right)\, dz}_{\mathfrak I}
\end{split}
	\label{gg2}
\end{equation}
The first two terms on the RHS of \eqref{gg2} can be bounded by
\begin{equation}
\label{gggg3}
 \norm{\mathcal{Q}}_0\left(\norm{g_3}_0+\norm{\mathcal V}_0\norm{\dfrac{r}{{\Rk}}\ak}_3\right) \leq CP\left(\sup_{t\in[0,T]}\mathfrak{E}^{\kappa}(t)\right).
\end{equation}
By the definition of $\mathcal{V}$ and $\mathcal Q$, since $\zeta^{\kappa}=\Lambda_\kappa^2\zeta$, $q=\hal\frac{C(t)^2}{\Lambda_\kappa^2R}$ on $\Gamma$, we have
\begin{equation}
\begin{split}
\mathfrak I&=\int_{\mathbb T} R_0(\bp^4q-\bp^4\Lambda_\kappa^2\zeta_j\ak_{j2}\partial_z q)\ak_{i1}\mathcal{V}_i\,dz-\int_{\bT}R_0\bp^4\Lambda_\kappa^2\zeta_j \ak_{j1}\partial_r q \ak_{i1}\mathcal{V}_i\,dz
\\&= \int_{\mathbb T} -R_0\dfrac{C(t)^2}{{\Rk}^3}\left(\bp^4\Lambda_\kappa^2 R-\bp^4\Lambda_\kappa^2\zeta_j\ak_{j2}\partial_z \Lambda_\kappa^2 R\right)\ak_{i1}\mathcal{V}_i\,dz-\int_{\bT}R_0\bp^4\Lambda_\kappa^2\zeta_j \ak_{j1}\partial_r q \ak_{i1}\mathcal{V}_i\,dz\\&\quad+\int_{\bT}R_0C(t)^2\left[\bp^3, \dfrac{1}{{\Rk}^3}\right]\bp \Lambda_{\kappa}^2 R\ak_{i1}\mathcal{V}_i\,dz
\end{split}
\end{equation}
The key observation here we have is the following: by using the definition of $\ak$, we have $\Jk \ak_{11}= \partial_z\Zk$, $\Jk \ak_{12}=-\partial_r\Zk$, $\Jk \ak_{21}=-\partial_z\Rk$, $\Jk \ak_{22}=\partial_r\Rk$, $\Jk=\partial_z\Zk\partial_r\Rk-\partial_R\Zk\partial_z\Rk$, and $\Rk=\Lambda_\kappa^2 R, \Zk=\Lambda_\kappa^2 Z$ on the boundary $\Gamma$, then
\begin{equation}
\label{req}
\begin{split}
\bp^4\Lambda_\kappa^2 R-\bp^4\Lambda_\kappa^2\zeta_j\ak_{j2}\partial_z \Lambda_\kappa^2 R&=\bp^4 \Lambda_\kappa^2 R\left(1-\ak_{12}\partial_z\Rk\right)+\bp^4\Lambda_\kappa^2 Z(-\ak_{22}\partial_z\Rk)
\\&=\bp^4 \Lambda_\kappa^2 R\left((\Jk)^{-1}\Jk+(\Jk)^{-1}\partial_r\Zk\partial_z\Rk\right)+\bp^4\Lambda_\kappa^2 Z(\ak_{22}\Jk\ak_{21})\\&=\bp^4 \Lambda_\kappa^2 R\left((\Jk)^{-1}\partial_z\Zk\partial_r\Rk\right)+\bp^4\Lambda_\kappa^2 Z(\ak_{22}\Jk\ak_{21})\\&=\bp^4 \Lambda_\kappa^2 R\left((\Jk)\ak_{11}\ak_{22}\right)+\bp^4\Lambda_\kappa^2 Z(\ak_{22}\Jk\ak_{21})\\&=\partial_r\Rk\bp^4\Lambda_\kappa^2\zeta_j\ak_{j1}
\end{split}
\end{equation}
With \eqref{req} and $\nu=\partial_t\zeta-\fk$, we then arrive at
\begin{equation}
\begin{split}
\mathfrak I&= \int_{\mathbb T}R_0\partial_r\left(\hal\dfrac{C(t)^2}{{\Rk}^2}-q\right)\left(\bp^4\Lambda_\kappa^2\zeta_j \ak_{j1}\right)\ak_{i1}\mathcal{V}_i\,dz\\&\quad+\int_{\bT}R_0C(t)^2\left[\bp^3, \dfrac{1}{{\Rk}^3}\right]\bp \Lambda_{\kappa}^2 R\ak_{i1}\mathcal{V}_i\,dz\\&= \int_{\mathbb T}R_0\partial_r\left(\hal\dfrac{C(t)^2}{{\Rk}^2}-q\right)\left(\bp^4\Lambda_\kappa^2\zeta_j \ak_{j1}\right)\ak_{i1}\left(\bp^4\partial_t\zeta_i-\bp^4\psi^{\kappa}_i-\bp^4\Lambda_\kappa^2\zeta_j\ak_{j\ell}\partial_\ell \nu_i\right)\,dz\\&\quad+\int_{\bT}R_0C(t)^2\left[\bp^3, \dfrac{1}{{\Rk}^3}\right]\bp \Lambda_{\kappa}^2 R\ak_{i1}\mathcal{V}_i\,dz
\end{split}
\end{equation}
Note that
\begin{equation}
\begin{split}
&\int_{\mathbb T}R_0\partial_r\left(\hal\dfrac{C(t)^2}{{\Rk}^2}-q\right)\left(\bp^4\Lambda_\kappa^2\zeta_j \ak_{j1}\right)\ak_{i1}\bp^4\partial_t\zeta_i \,dz \\
 =&\int_{\bT}R_0\partial_r\left(\hal\dfrac{C(t)^2}{{\Rk}^2}-q\right)\ak_{j1}\bp^4 \Lambda_\kappa \zeta_j  \ak_{i1} \bp^4 \Lambda_\kappa \partial_t\zeta_i\,dz
 \\&\qquad+ \int_{\bT}\bp^4 \Lambda_\kappa \zeta_j \left[\Lambda_\kappa,  R_0\partial_r\left(\hal\dfrac{C(t)^2}{{\Rk}^2}-q\right) \ak_{j1}\ak_{i1}\right] \bp^4 \partial_t\zeta_i \,dz\\
=&\hal \frac{d}{dt}\int_{\bT}R_0\partial_r\left(\hal\dfrac{C(t)^2}{{\Rk}^2}-q\right)\abs{\ak_{i1}\bp^4 \Lambda_\kappa \zeta_i  }^2\,dz  -\hal \int_{\bT}R_0\dt\partial_r\left(\hal\dfrac{C(t)^2}{{\Rk}^2}-q\right) \abs{\ak_{i1}\bp^4 \Lambda_\kappa \zeta_i  }^2\,dz
\\&\qquad -\int_{\bT} R_0\partial_r\left(\hal\dfrac{C(t)^2}{{\Rk}^2}-q\right)  \ak_{j1}\bp^4 \Lambda_\kappa \zeta_j  \dt\ak_{i1} \bp^4  \Lambda_\kappa \zeta_i\,dz
\\&\qquad +\int_{\bT}\bp^4 \Lambda_\kappa \zeta_j \left[\Lambda_\kappa, R_0\partial_r\left(\hal\dfrac{C(t)^2}{{\Rk}^2}-q\right) \ak_{j1}\ak_{i1}\right] \bp^4 \partial_t\zeta_i\,dz .
\end{split}
\end{equation}
Therefore, we obtain
\begin{equation}
\label{gg3}
\begin{split}
& \int_{\mathbb T}R_0\partial_r\left(\hal\dfrac{C(t)^2}{{\Rk}^2}-q\right)\left(\bp^4\Lambda_\kappa^2\zeta_j \ak_{j1}\right)\ak_{i1}\mathcal{V}_i\,dz\\=&\hal \frac{d}{dt}\int_{\bT}R_0\partial_r\left(\hal\dfrac{C(t)^2}{{\Rk}^2}-q\right)\abs{\ak_{i1}\bp^4 \Lambda_\kappa \zeta_i  }^2\,dz  -\hal \int_{\bT}R_0\dt\partial_r\left(\hal\dfrac{C(t)^2}{{\Rk}^2}-q\right) \abs{\ak_{i1}\bp^4 \Lambda_\kappa \zeta_i  }^2\,dz
\\&\qquad -\int_{\bT} R_0\partial_r\left(\hal\dfrac{C(t)^2}{{\Rk}^2}-q\right)  \ak_{j1}\bp^4 \Lambda_\kappa \zeta_j  \dt\ak_{i1} \bp^4  \Lambda_\kappa \zeta_i\,dz
\\&\qquad +\int_{\bT}\bp^4 \Lambda_\kappa \zeta_j \left[\Lambda_\kappa, R_0\partial_r\left(\hal\dfrac{C(t)^2}{{\Rk}^2}-q\right) \ak_{j1}\ak_{i1}\right] \bp^4 \partial_t\zeta_i\,dz
 \\
=&\hal \frac{d}{dt}\int_{\bT}R_0\partial_r\left(\hal\dfrac{C(t)^2}{{\Rk}^2}-q\right)\abs{\ak_{i1}\bp^4 \Lambda_\kappa \zeta_i  }^2\,dz
-\underbrace{\hal \int_{\bT}R_0\dt\partial_r\left(\hal\dfrac{C(t)^2}{{\Rk}^2}-q\right) \abs{\ak_{i1}\bp^4 \Lambda_\kappa \zeta_i  }^2\,dz}_{\mathcal{I}_1}
\\&\qquad -\underbrace{\int_{\bT} R_0\partial_r\left(\hal\dfrac{C(t)^2}{{\Rk}^2}-q\right)  \ak_{j1}\bp^4 \Lambda_\kappa \zeta_j  \dt\ak_{i1} \bp^4  \Lambda_\kappa \zeta_i\,dz }_{\mathcal{I}_2}
\\&\qquad+\underbrace{\int_{\bT}\bp^4 \Lambda_\kappa \zeta_j \left[\Lambda_\kappa, R_0\partial_r\left(\hal\dfrac{C(t)^2}{{\Rk}^2}-q\right) \ak_{j1}\ak_{i1}\right] \bp^4 \partial_t\zeta_i\,dz}_{\mathcal{I}_3}
\\&\qquad -\underbrace{\int_{\bT}R_0\partial_r\left(\hal\dfrac{C(t)^2}{{\Rk}^2}-q\right)\ak_{j1}\bp^4 \Lambda_\kappa^2\zeta_j  \ak_{i1}  \bp^4 \Lambda_\kappa^2\zeta \cdot\nak \nu_i   }_{\mathcal{I}_4}
 \\&\qquad -\underbrace{\int_{\Gamma} R_0\partial_r\left(\hal\dfrac{C(t)^2}{{\Rk}^2}-q\right) \ak_{j1}\bp^4 \Lambda_\kappa^2\zeta_j  \ak_{i1}  \bp^4 \psi^\kappa_i  }_{\mathcal{I}_5}.
\end{split}
\end{equation}

We now estimate $\mathcal{I}_1$--$\i_5$. By the estimates \eqref{pressure}, we deduce
\begin{equation}
\label{i0}
\mathcal{I}_1 \ls \left(\abs{\partial_r\partial_t q}_{L^{\infty}}+\abs{\partial_r R}_{L^{\infty}}\right)\abs{\ak_{i1}\bp^4 \Lambda_\kappa \zeta_i  }^2\ls \left(\norm{\partial_tq}_3+\norm{R}_4\right)\abs{\ak_{i1}\bp^4 \Lambda_\kappa \zeta_i  }^2\le P\left(\sup_{t\in[0,T]} \mathfrak{E}^{\kappa}(t)\right).
\end{equation}
By the identity \eqref{partialF}, we have
\begin{align}
\nonumber \i_2&= \int_{\bT} R_0\partial_r\left(\hal\dfrac{C(t)^2}{{\Rk}^2}-q\right) \a^\kappa_{j1}\bp^4 \Lambda_\kappa \zeta_j   \a^{\kappa}_{i\ell}\pa_{a_\ell}\dt\zeta^{\kappa}_m\ak_{m1} \bp^4  \Lambda_\kappa \zeta_i
\\&= \underbrace{\int_{\bT} R_0\partial_r\left(\hal\dfrac{C(t)^2}{{\Rk}^2}-q\right) \a^\kappa_{j1}\bp^4 \Lambda_\kappa \zeta_j   \a^{\kappa}_{i1}\pa_r\dt\zeta^{\kappa}_m\a^\kappa_{m1} \bp^4  \Lambda_\kappa \zeta_i   }_{\i_{2a}}\nonumber
\\&\quad+\int_{\bT}  R_0\partial_r\left(\hal\dfrac{C(t)^2}{{\Rk}^2}-q\right) \a^\kappa_{j1}\bp^4 \Lambda_\kappa \zeta_j   \a^{\kappa}_{i2}\pa_z\dt\Lambda_{\kappa}^2\zeta_m\a^\kappa_{m1} \bp^4  \Lambda_\kappa \zeta_i.\label{tmp1}
\end{align}
As usual, we obtain
\begin{equation}
\label{1a}
\i_{2a} \ls\abs{\a^\kappa_{j1}\bp^4 \Lambda_\kappa \zeta_j  }_0^2\abs{R_0\partial_r\left(\hal\dfrac{C(t)^2}{{\Rk}^2}-q\right) \pa_r\dt\zeta^{\kappa}_m\a^\kappa_{m1}}_{L^{\infty}}\le P\left(\sup_{t\in[0,T]} \mathfrak{E}^{\kappa}(t)\right).
\end{equation}
On the other hand, using $\dt\zeta=\nu+\fk$ we have
\begin{align}
&\int_{\bT} R_0\partial_r\left(\hal\dfrac{C(t)^2}{{\Rk}^2}-q\right) \a^\kappa_{j1}\bp^4 \Lambda_\kappa \zeta_j   \nonumber \a^{\kappa}_{i2}\pa_z\dt\Lambda_{\kappa}^2\zeta_m\a^\kappa_{m1} \bp^4  \Lambda_\kappa \zeta_i
\\&\quad=\underbrace{\int_{\bT}  R_0\partial_r\left(\hal\dfrac{C(t)^2}{{\Rk}^2}-q\right) \a^\kappa_{j1}\bp^4 \Lambda_\kappa \zeta_j   \a^{\kappa}_{i2}\pa_z \Lambda_{\kappa}^2 \nu_m \a^\kappa_{m1} \bp^4  \Lambda_\kappa \zeta_i   }_{\i_{2b}}\nonumber
\\&\qquad+\underbrace{\int_{\bT}  R_0\partial_r\left(\hal\dfrac{C(t)^2}{{\Rk}^2}-q\right) \a^\kappa_{j1}\bp^4 \Lambda_\kappa \zeta_j   \a^{\kappa}_{i2}\pa_z \Lambda_{\kappa}^2 \psi^\kappa_m\a^\kappa_{m1} \bp^4  \Lambda_\kappa \zeta_i   }_{\i_{2c}}
\label{ttt}
\end{align}
To estimate $\i_{2c}$, the difficulty is that one can not have an $\kappa$-independent control of $\abs{\bp^4  \Lambda_\kappa \zeta_i }_0$. Our observation is that since $\fk\rightarrow 0$ as $\kappa\rightarrow 0$, this motives us to deduce the following estimates:
 \begin{equation}\label{lwuqing}
 \abs{\bp   \fk}_{L^{\infty} }\leq \sqrt\kappa P(\norm{\zeta}_4,\norm{\nu}_3).
 \end{equation}
Indeed, we can rewrite the boundary condition in \eqref{etaaa} as
\begin{align}
 \fk=\int_0^z\int_0^\tau\mathbb{P} f^\kappa,\ f^\kappa:= \bp^2 (\zeta_j-{\Lambda_{\kappa}^2\zeta_j})\a^{\kappa}_{j2}\partial_{z}{\Lambda_{\kappa}^2 v}-\bp^2{\Lambda_{\kappa}^2\zeta}_j\a^{\kappa}_{j2}\partial_{z}(\nu-{\Lambda_{\kappa}^2 \nu}) .
\end{align}
By using Morrey's inequality and the Sobolev embeddings and the trace theorem,
\begin{equation}
\abs{g-\Lambda_{\kappa}g}_{L^\infty}\ls \sqrt{\kappa}\abs{\bp g}_{L^4} \ls \sqrt{\kappa}\abs{g}_{1}\ls \sqrt{\kappa}\norm{g}_{2},
\end{equation}
we obtain
\begin{equation}
\begin{split}
\abs{f^\kappa}_{L^\infty}&\ls\abs{\a^{\kappa}_{j2}\partial_{z}\Lambda_{\kappa}^2 \nu}_{L^\infty}\abs{\bp^2\zeta-\Lambda_{\kappa}^2\bp^2\zeta}_{L^\infty}+ \abs{\bp^2{\Lambda_{\kappa}^2\zeta}_j\a^{\kappa}_{j2}}_{L^\infty}\abs{\partial_{z}\nu-\Lambda_{\kappa}^2 \partial_{z}\nu}_{L^\infty} \\&\ls\sqrt\kappa P(\norm{\zeta}_4,\norm{\nu}_3).
\end{split}
\end{equation}
Then by the elliptic estimate and the Sobolev embeddings, we deduce
 \begin{equation}
\abs{\bp \fk}_{L^{\infty} }\ls\abs{\bp \fk}_1\ls\abs{f^\kappa}_{0}\ls\abs{f^\kappa}_{L^\infty} \ls \sqrt\kappa P(\norm{\zeta}_4,\norm{\nu}_3),
 \end{equation}
 which proves \eqref{lwuqing}. Hence, by \eqref{lwuqing} together with \eqref{loss}, we have
\begin{equation}\label{22c}
\begin{split}
\i_{2c}&\ls \abs{\partial_r\left(\hal\dfrac{C(t)^2}{{\Rk}^2}-q\right)\a^\kappa_{m1}\a^{\kappa}_{i2}}_{L^{\infty} }\abs{\a^\kappa_{j1}\bp^4 \Lambda_\kappa \zeta_j }_0\abs{\bp^4 \Lambda_\kappa \zeta_i}_0\abs{\pa_z \Lambda_{\kappa}^2 \psi^\kappa_m}_{L^{\infty} }
\\ &\ls\abs{\partial_r\left(\hal\dfrac{C(t)^2}{{\Rk}^2}-q\right)\a^\kappa_{m1}\a^{\kappa}_{i2}}_{L^{\infty} }\abs{\a^\kappa_{j1}\bp^4 \Lambda_\kappa \zeta_j }_0\dfrac{1}{\sqrt{\kappa}}\abs{\zeta}_{7/2}\sqrt{\kappa}P(\norm{\zeta}_4,\norm{\nu}_3)
\\&\leq P\left(\sup_{t\in[0,T]} \mathfrak{E}^{\kappa}(t)\right).
\end{split}
\end{equation}
Note that the term $\i_{2b}$ is out of control by an $\kappa$-independent bound alone.

For $\i_3$, by the commutator estimates \eqref{es1-1/2}, \eqref{co123}, \eqref{pressure}, \eqref{tes1} and \eqref{fest1}, we obtain
\begin{align}
\nonumber
\i_3&\leq \abs{\bp^4 \Lambda_\kappa \zeta_j }_{-1/2}\abs{\left[\Lambda_\kappa, R_0\partial_r\left(\hal\dfrac{C(t)^2}{{\Rk}^2}-q\right)\a^\kappa_{j1}\a^{\kappa}_{i1}\right] (\bp^4 \partial_t\zeta_i)}_{1/2}
\\\nonumber &\ls \abs{\bp^3\Lambda_\kappa \zeta_j }_{1/2}\abs{R_0\partial_r\left(\hal\dfrac{C(t)^2}{{\Rk}^2}-q\right) \a^{\kappa}_{j1}\a^{\kappa}_{i1}}_{W^{1,\infty} }\abs{\bp^3\partial_t\zeta}_{1/2}\nonumber
\\&\ls \norm{\zeta}_{4}\norm{\partial_r\left(\hal\dfrac{C(t)^2}{{\Rk}^2}-q\right)\a^{\kappa}_{j1}\a^{\kappa}_{i1}}_3\norm{\nu+\psi^{\kappa}}_{4} \leq P\left(\sup_{t\in[0,T]} \mathfrak{E}^{\kappa}(t)\right).
\label{commm1}
\end{align}

To control $\i_4$, similarly as \eqref{tmp1}, we write
\begin{align}
\nonumber
\i_4=&\underbrace{\int_{\bT}R_0\partial_r\left(\hal\dfrac{C(t)^2}{{\Rk}^2}-q\right) \a^\kappa_{j1}\bp^4 \Lambda_\kappa^2\zeta_j  \a^{\kappa}_{m1}  \bp^4\Lambda_\kappa^2\zeta_{i}\a^{\kappa}_{i1}\partial_r \nu_m   }_{\i_{4a}}\\&+\underbrace{\int_{\bT}R_0\partial_r\left(\hal\dfrac{C(t)^2}{{\Rk}^2}-q\right) \a^\kappa_{j1}\bp^4 \Lambda_\kappa^2\zeta_j \a^{\kappa}_{i2} \partial_{z} \nu_m\a^{\kappa}_{m1} \bp^4\Lambda_\kappa^2\zeta_i  }_{\i_{4b}}.\label{ttttt}
\end{align}
By the commutator estimates \eqref{es1-0}, we have
\begin{equation}
\label{ennennene}
\begin{split}
\abs{\a^\kappa_{j1}\bp^4 \Lambda_\kappa^2\zeta_j }_0&\ls \abs{[\Lambda_{\kappa}, \a^\kappa_{j1}](\bp^4\Lambda_{\kappa}\eta_j)}_0+\abs{\a^\kappa_{j1}\bp^4(\Lambda_\kappa\zeta_j)}_0
\\&\ls\abs{\a^{\kappa}}_{W^{1,\infty} }\abs{\bp^3\Lambda_{\kappa}\zeta_j}_0+\abs{\a^\kappa_{j1}\bp^4(\Lambda_\kappa\zeta_j)}_0.
\end{split}
\end{equation}
Then we obtain
\begin{equation}\label{3b}
\i_{4a}\ls
\left(\abs{\a^\kappa_{j1}\bp^4 \Lambda_\kappa^2\zeta_j }_0^2+\abs{\ak}_{W^{1,\infty}}^2\norm{\zeta}_4^2\right)\abs{\partial_r\left(\hal\dfrac{C(t)^2}{{\Rk}^2}-q\right)\a^{\kappa}_{m1}\partial_r \nu_m}_{L^{\infty} }
 \leq P\left(\sup_{t\in[0,T]} \mathfrak{E}^{\kappa}(t)\right).
\end{equation}
Note that the term $\i_{4b}$ is also out of control by an $\kappa$-independent bound alone.

Now we take care of $\i_{2b}$ and $\i_{4b}$. Notice that $\mathcal{I}_{2b}$ and $\mathcal{I}_{4b}$ are cancelled out in the limit $\kappa\rightarrow0$, however, it is certainly not the case when $\kappa>0$. This is most involved thing in the tangential energy estimates. Note also that we can not use the commutator estimate to interchange the position of the mollifier operator $\Lambda_\kappa$ in each of two terms since  $\abs{\bp^4\zeta}_{L^\infty}$ is out of control.
The key point here is to use the term $\i_5$, by the definition of the modification term $\fk$, to kill out both $\i_{2b}$ and $\i_{4b}$; this is exactly the reason that we have introduced $\fk$. By the boundary condition in \eqref{etaaa}, we deduce
\begin{equation}
\begin{split}
\i_5&=\int_{\bT}R_0\partial_r\left(\hal\dfrac{C(t)^2}{{\Rk}^2}-q\right) \a^\kappa_{j1}\bp^4 \Lambda_\kappa^2\zeta_j  \a^{\kappa}_{i1}  \bp^2 \left( \bp^2\zeta_m \a^\kappa_{m2}\pa_z{\Lambda_{\kappa}^2 \nu_i}-\bp^2\Lambda_{\kappa}^2\zeta_m \a^\kappa_{m2}\pa_z{ \nu_i}\right)
\\&=\int_{\bT} R_0\partial_r\left(\hal\dfrac{C(t)^2}{{\Rk}^2}-q\right)   \a^\kappa_{j1}\bp^4 \Lambda_\kappa^2\zeta_j  \a^{\kappa}_{i1}    \bp^4\zeta_m \a^\kappa_{m2}\pa_z{\Lambda_{\kappa}^2 \nu_i}
\\&\quad+\underbrace{\int_{\bT} -R_0\partial_r\left(\hal\dfrac{C(t)^2}{{\Rk}^2}-q\right)  \a^\kappa_{j1}\bp^4 \Lambda_\kappa^2\zeta_j  \a^{\kappa}_{i1}
\bp^4\Lambda_{\kappa}^2\zeta_m \a^\kappa_{m2}\pa_z{\nu_i} }_{-\i_{4b}}
\\&\quad+\underbrace{\int_{\Gamma} R_0\partial_r\left(\hal\dfrac{C(t)^2}{{\Rk}^2}-q\right)  \a^\kappa_{j1}\bp^4 \Lambda_\kappa^2\zeta_j  \a^{\kappa}_{i1}  \left( \left[\bp^2, \a^\kappa_{m2}\pa_z{\Lambda_{\kappa}^2 \nu_i}\right]\bp^2\zeta_m -\left[\bp^2,\a^\kappa_{m2}\pa_z{ \nu_i}\right] \bp^2\Lambda_{\kappa}^2\zeta_m \right)  }_{\i_{5a}}
\label{tmp2}
\end{split}
\end{equation}
By doing estimates as usual and using \eqref{ennennene} again, we have
\begin{equation}
\label{2a}
\begin{split}
\i_{5a} &\ls |\partial_r\left(\hal\dfrac{C(t)^2}{{\Rk}^2}-q\right)\a^{\kappa}_{i1}|_{L^{\infty}}\abs{\a^\kappa_{j1}\bp^4 \Lambda_\kappa^2\zeta_j }_0\abs{\left(\left[\bp^2, \a^\kappa_{m2}\pa_z{\Lambda_{\kappa}^2 \nu_i}\right]\bp^2\zeta_m -\left[\bp^2,\a^\kappa_{m2}\pa_z{ \nu_i}\right] \bp^2\Lambda_{\kappa}^2\zeta_m \right)}_0
\\
&\le P\left(\sup_{t\in[0,T]} \mathfrak{E}^{\kappa}(t)\right).
\end{split}
\end{equation}
We rewrite the first term as
\begin{align}
\nonumber&\int_{\bT} R_0\partial_r\left(\hal\dfrac{C(t)^2}{{\Rk}^2}-q\right)   \a^\kappa_{j1}\bp^4 \Lambda_\kappa^2\zeta_j  \a^{\kappa}_{i1}    \bp^4\zeta_m \a^\kappa_{m2}\pa_z{\Lambda_{\kappa}^2 \nu_i}
\\\nonumber&\quad=\underbrace{\int_{\bT} R_0\partial_r\left(\hal\dfrac{C(t)^2}{{\Rk}^2}-q\right)  \a^\kappa_{j1}\bp^4 \Lambda_\kappa \zeta_j  \a^{\kappa}_{i1}    \bp^4 \Lambda_\kappa \zeta_m \a^\kappa_{m2}\pa_z{\Lambda_{\kappa}^2 \nu_i}}_{-\i_{2b}}
\\ &\qquad+\underbrace{\int_{\bT}\bp^4 \Lambda_\kappa \zeta_j \left[\Lambda_\kappa,  R_0\partial_r\left(\hal\dfrac{C(t)^2}{{\Rk}^2}-q\right) \a^\kappa_{j1}\a^{\kappa}_{i1} \a^\kappa_{m2}\pa_z{\Lambda_{\kappa}^2 \nu_i}\right]\bp^4\zeta_m   }_{\i_{5b}}.\label{temp3}
\end{align}
By arguing similarly as \eqref{commm1} for $\i_3$, we have
\begin{equation}
\label{pw}
\i_{5b}\le \norm{\zeta}_{4}\norm{\partial_r\left(\hal\dfrac{C(t)^2}{{\Rk}^2}-q\right)\a^\kappa_{j1}\a^{\kappa}_{i1} \a^\kappa_{m2}\pa_z{\Lambda_{\kappa}^2 \nu_i}}_3\norm{\zeta}_{4}\leq P\left(\sup_{t\in[0,T]} \mathfrak{E}^{\kappa}(t)\right).
\end{equation}

Now combining \eqref{tmp1}, \eqref{ttt}, \eqref{ttttt}, \eqref{tmp2} and \eqref{temp3},  and using the estimates \eqref{1a}, \eqref{22c}, \eqref{3b}, \eqref{2a} and \eqref{pw}, we deduce
\begin{equation}\label{ggiip}
\i_2+\i_4+\i_5=\i_{2a}+\i_{2c}+\i_{4a}+\i_{5a}+\i_{5b}\le P\left(\sup_{t\in[0,T]} \mathfrak{E}^{\kappa}(t)\right).
\end{equation}

Finally, combining \eqref{gg1}, \eqref{gg2} and \eqref{gg3}, and using the estimates \eqref{ggg0}, \eqref{gggg3}, \eqref{commm1} and \eqref{ggiip}, we obtain
\begin{equation}\label{ohoh}
 \dfrac{d}{dt}\left(\int_{\Omega}\abs{\mathcal{V}}^2+\abs{\bp^4( b_0\cdot\nabla\zeta)}^2 + \int_{\Gamma}\partial_r\left(\hal\dfrac{C(t)^2}{{\Rk}^2}-q\right)\abs{\bp^4 \Lambda_\kappa \zeta_i \a^\kappa_{i1} }^2   \right)\leq  P\left(\sup_{t\in[0,T]} \mathfrak{E}^{\kappa}(t)\right).
\end{equation}
Integrating \eqref{ohoh} directly in time, by the a priori assumption \eqref{ini2}, we have
\begin{equation}\label{ohoh1}
\norm{\mathcal V(t)}_0^2+\norm{\bp^4 (b_0\cdot\nabla \zeta)(t)}_0^2+\abs{\bp^4\Lambda_{\kappa}\zeta_i \a^{\kappa}_{i1}(t)}_0^2\leq M_0+TP\left(\sup_{t\in[0,T]} \mathfrak{E}^{\kappa}(t)\right).
\end{equation}
By the definition of $\mathcal{V}$, using \eqref{ohoh1} and  \eqref{etaest}, using the fundamental theorem of calculous, we get
\begin{equation}
\begin{split}
\norm{\bp^4 \nu(t)}_0^2&\ls \norm{\bp^4 \nu(t)}_0^2\ls\norm{\mathcal V(t)}_0^2+\norm{\bp^4\zeta}_0^2\norm{\a^{\kappa}_{j\ell}\partial_{a_\ell} \nu}_{L^{\infty}}^2
\\&\leq M_0+TP\left(\sup_{t\in[0,T]} \mathfrak{E}^{\kappa}(t)\right).
\end{split}
\end{equation}
We thus conclude the proposition.
\end{proof}
\subsection{Normal estimates for $\nu=(v^r, v^z)$}\label{curle}
In this subsection, we control the normal derivatives of $\nu=(v^r, v^z)$ by using the equaton \eqref{divValpha}, and the curl equation
\begin{equation}
	\begin{split}
	&	\partial_t(\pa^{\ak}_Z v^r-\pa^{\ak}_R v^z)-b_0\cdot\nabla\left(\pa^{\ak}_Z(b_0\cdot\nabla R)-\pa^{\ak}_R(b_0\cdot\nabla Z)\right)\\&=\left[\pa^{\ak}_Z, b_0\cdot\nabla\right](b_0\cdot\nabla R)-\left[\pa^{\ak}_R, b_0\cdot\nabla\right](b_0\cdot\nabla Z)+\left[\partial_t,\pa^{\ak}_Z\right] v^r-\left[\partial_t,\pa^{\ak}_R\right] v^z
\end{split}
	\label{curlequ}
\end{equation}
We denote $\curl \nu:=\partial_z v^r-\partial_r v^z, \curl b:=\partial_z(b_0\cdot\nabla R)-\partial_r(b_0\cdot\nabla Z)$, $\curl_{\ak}\nu:=\partial_Z^{\ak}v^r-\partial_R^{\ak}v^z$, $\curl_{\ak} b:=\partial_Z^{\ak}(b_0\cdot\nabla R)-\partial_R^{\ak}(b_0\cdot\nabla Z)$ and begin with the energy estimates for \eqref{curlequ}.
\begin{proposition}
	\label{curlprop}
For $t\in [0,T]$ with $T\le T_\kappa$, it holds that
\begin{equation}
\label{curlest}
\norm{\curl v(t)}_{3}^2+\norm{\curl b(t)}_{3}^2\leq M_0+CTP\left(\sup_{t\in[0,T]} \mathfrak{E}^{\kappa}(t)\right).
\end{equation}
\end{proposition}
\begin{proof}
Apply $D^3$ to \eqref{curlequ} to get
\begin{equation}\label{curl2}
	\partial_t(D^3(\curl_{\ak} v))-b_0\cdot\nabla\left(D^3\left(\curl_{\ak} b\right)\right)=F,
\end{equation}
with
\begin{equation*}
	\begin{split}
	F:=&[D^3,b_0\cdot\nabla]\left(\curl_{\ak} b\right)\\&+D^3\left([\pa^{\ak}_Z, b_0\cdot\nabla](b_0\cdot\nabla R)-[\pa^{\ak}_R, b_0\cdot\nabla](b_0\cdot\nabla Z)+[\partial_t,\pa^{\ak}_Z] v^r-[\partial_t,\pa^{\ak}_R] v^z\right).
\end{split}
\end{equation*}
Taking the $L^2$ inner product of \eqref{curl2} with $D^3(\pa^{\ak}_Z v^r-\pa^{\ak}_R v^z)$, by the integration by parts, we get
\begin{align}\label{j00}
	&\dfrac{1}{2}\dfrac{d}{dt}\int_{\Omega}\abs{D^3(\curl_{\ak} v)}^2\,dx+\underbrace{\int_{\Omega}D^3(\curl_{\ak}b) D^3(\pa^{\ak}_Z (b_0\cdot\nabla v^r)-\pa^{\ak}_R(b_0\cdot\nabla v^z))}_{\mathcal{J}_1}\\&\quad=\underbrace{\int_\Omega F D^3(\pa^{\ak}_Z v^r-\pa^{\ak}_R v^z)}_{\j_2}+\underbrace{\int_{\Omega}D^{3}(\curl_{\ak}b)(\left[D^3\pa^{\ak}_Z, b_0\cdot\nabla\right] v^r-\left[D^3\pa^{\ak}_Z, b_0\cdot\nabla\right]v^z)}_{\j_3}.\nonumber
\end{align}
Since $\nu=\dt\zeta-\fk$, we have,
\begin{equation}
\label{j0}
\begin{split}
	\j_1=&\dfrac{1}{2}\dfrac{d}{dt}\int_{\Omega}\abs{D^3\left(\curl_{\ak}b\right)}^2\,dx
	-\underbrace{\dfrac{1}{2}\int_{\Omega}D^3\left(\curl_{\ak}b\right)D^3\left(\partial_t\ak D(b_0\cdot\nabla \zeta)\right)}_{\j_{1a}}\\&-\underbrace{\int_{\Omega}D^3\curl_{\a^{\kappa}}b D^3\left(\pa_Z^{\ak}(b_0\cdot\nabla\fk_1)-\pa_R^{\ak}(b_0\cdot\nabla\fk_2)\right)}_{\j_{1b}}.
\end{split}
\end{equation}
By the identity \eqref{partialF} and the estimates \eqref{tes0} and \eqref{fest1}, we have
\begin{equation*}
\j_{1a}
\ls  \norm{\ak}_3\norm{D(b_0\cdot\nabla \zeta)}_3\norm{\partial_t\ak}_3\norm{D(b_0\cdot\nabla \zeta)}_3
\le CP\left(\sup_{t\in[0,T]} \mathfrak{E}^{\kappa}(t)\right).
\end{equation*}
By the estimates \eqref{fest22} and \eqref{tes1}, we obtain
\begin{equation}
\label{j0b}
\j_{1b}\ls \norm{b_0\cdot\nabla\zeta}_4\norm{\a^{\kappa}}_3^2\norm{b_0\cdot\nabla\fk}_4\leq P\left(\sup_{t\in[0,T]} \mathfrak{E}^{\kappa}(t)\right).
\end{equation}
Hence, we arrive at
\begin{equation}
\label{j1}
\j_1\ge \dfrac{1}{2}\dfrac{d}{dt}\int_{\Omega}\abs{D^3\left(\curl_{\ak}b\right)}^2
-CP\left(\sup_{t\in[0,T]} \mathfrak{E}^{\kappa}(t)\right).
\end{equation}
We now turn to estimate the right hand side of \eqref{j00}. By the identity \eqref{partialF}, we may have
\begin{equation}
\label{j2}
\begin{split}
	\j_2\leq  \norm{F}_0 \norm{\curl_{\ak}v}_3
	\leq CP\left(\sup_{t\in[0,T]} \mathfrak{E}^{\kappa}(t)\right).
\end{split}
\end{equation}
Similarly,
\begin{equation}
	\label{j3}
	\begin{split}
	\j_3&\ls \norm{D^{3}(\curl_{\ak}b))}_0\norm{\left[D^3\pa^{\ak}_Z, b_0\cdot\nabla\right] v^r-\left[D^3\pa^{\ak}_Z, b_0\cdot\nabla\right]v^z}_0\\&
\leq CP\left(\sup_{t\in[0,T]} \mathfrak{E}^{\kappa}(t)\right).
\end{split}
\end{equation}
Consequently, plugging the estimates \eqref{j1}--\eqref{j3} into \eqref{j00}, we obtain
\begin{equation} \label{kllj}
	\dfrac{d}{dt}\int_{\Omega}\abs{D^3(\curl_{\ak}v)}^2+\abs{D^3\left(\curl_{\ak}b\right)}^2\,dx
		\le CP\left(\sup_{t\in[0,T]} \mathfrak{E}^{\kappa}(t)\right).
\end{equation}
Integrating \eqref{kllj} directly in time, and applying the fundamental theorem of calculous,
\begin{equation}
	\norm{\curl  f(t)}_3\le \norm{\curl_{\ak}f(t)}_3+\norm{\int_0^t\partial_t\ak d\tau D f(t)}_3,
\end{equation}
we then conclude the proposition.
\end{proof}
We now derive the divergence estimates. We denote $\Div v:=\partial_rv^r+\dfrac{1}{r}v^r+\partial_zv^z, \Div b:=\partial_r(b_0\cdot\nabla R)+\dfrac{1}{r}(b_0\cdot\nabla R)+\partial_z(b_0\cdot\nabla Z)$.
\begin{proposition}
	\label{divprop}
For $t\in [0,T]$ with $T$, it holds that
\begin{equation}
\label{divest}
\norm{\Div \nu(t)}_3^2+\norm{\Div b(t)}_3^2 \leq TP\left(\sup_{t\in[0,T]} \mathfrak{E}^{\kappa}(t)\right).
\end{equation}
\end{proposition}
\begin{proof}
	From the third equation of \eqref{sub1} and $\ak|_{t=0}=I$, we see that
\begin{equation}
	\label{fpe}
	\partial_r v^r+\dfrac{1}{r}v^r+\partial_z v^z=-\int_0^t\partial_t\ak_{ij}\,d\tau\partial_{a_j}\nu_i
	+\int_0^t\dfrac{r\partial_t{\Rk}}{{\Rk}^2}\,d\tau\dfrac{v^r}{r}.
\end{equation}
Hence, it is clear that by the identity \eqref{partialF}, and the estimates \eqref{tes1} and \eqref{tes0},
\begin{equation}\label{divv1}
\norm{\partial_r v^r+\dfrac{1}{r}v^r+\partial_z v^z}_{3}^2\le TP\left(\sup_{t\in[0,T]} \mathfrak{E}^{\kappa}(t)\right).
\end{equation}
From the third equation of \eqref{sub1} again, we have
\begin{equation*}
	\Div_{\ak}b_0\cdot\nabla \nu=\left[\Div_{\ak},b_0\cdot\nabla\right]\nu.
\end{equation*}
This together with the equation $\nu=\partial_t \zeta-\psi^\kappa$, we have
\begin{equation}
\label{divb}
\begin{split}
	\partial_t\left(\Div_{\ak}(b_0\cdot\nabla \zeta)\right)=&\Div_{\ak}
(b_0\cdot\nabla\fk)+\left[\Div_{\ak},b_0\cdot\nabla\right]\nu+ \partial_t \ak_{i\ell} \partial_{a_\ell}\left(b_0\cdot\nabla \zeta_i\right)+\partial_t\dfrac{1}{{\Rk}}b_0\cdot\nabla R.
\end{split}
\end{equation}
This implies that, by doing the $D^3$ energy estimate and using the estimates \eqref{tes1}--\eqref{fest22} and the identity \eqref{partialF},
\begin{equation}\label{di2}
\norm{\Div_{\ak}(b_0\cdot\nabla \zeta)}_{3}^2
\leq  CTP\left(\sup_{t\in[0,T]} \mathfrak{E}^{\kappa}(t)\right).
\end{equation}
And then applying the fundamental theorem of calculous, and a similar equality as \eqref{fpe}, we arrive at
\begin{equation}\label{divv12}
\norm{\Div b}_{3}^2
\leq  TP\left(\sup_{t\in[0,T]} \mathfrak{E}^{\kappa}(t)\right).
\end{equation}

Consequently, we conclude the proposition by the estimates \eqref{divv1} and \eqref{divv12}.
\end{proof}
Now we show how to get normal derivatives of $\nu=(v^r, v^z)$ and $b_0\cdot\nabla\zeta$ by using Proposition \ref{curlprop} and Proposition \ref{divprop}:
\begin{proposition}\label{nore}
For $t\in [0,T]$, it holds that
\begin{equation}
\label{teedde}
\norm{D^4 \nu (t)}_0^2+\norm{D^4(b_0\cdot\nabla \zeta)(t)}_0^2\leq M_0+CTP\left(\sup_{t\in[0,T]}\mathfrak{E}^{\kappa}(t)\right).
\end{equation}
\end{proposition}
\begin{proof}
First, by using the Proposition \ref{curlprop} and the Proposition \ref{te}, we have
\begin{equation}
\label{fefe}
\begin{split}
&\norm{\partial_z^3\partial_rv^z}_0^2\leq M_0+CTP\left(\sup_{t\in[0,T]} \mathfrak{E}^{\kappa}(t)\right),\\
&\norm{\partial_z^3(\partial_rv^r+\dfrac{1}{r}v^r)}_0^2\leq M_0+CTP\left(\sup_{t\in[0,T]} \mathfrak{E}^{\kappa}(t)\right).
\end{split}
\end{equation}
By direct calculation, we have
\begin{equation}
	\norm{\partial_z^3\partial_r v^r+\partial_z^3(\dfrac{v^r}{r})}_0^2=\norm{\partial_z^3\partial_r v^r}_0^2+\norm{\partial_z^3(\dfrac{v^r}{r})}_0^2+2\int_{\Omega}\partial_z^3\partial_r(r\dfrac{v^r}{r})\partial_z^3(\dfrac{v^r}{r})\,dx
	\label{sln}
\end{equation}
and by integration-by-parts, the last term of the above equality can be calculated as
\begin{equation}	2\int_{\Omega}\partial_z^3\partial_r(r\dfrac{v^r}{r})\partial_z^3(\dfrac{v^r}{r})\,dx=\int_{\Omega}|\partial_z^3(\dfrac{v^r}{r})|^2\,dx+\int_{\mathbb T}\abs{\partial_z^3(\dfrac{v^r}{r})}^2(R_0,z)R_0^2\,dz
	\label{sfo}
\end{equation}
Thus, we arrive at
\begin{equation}
	\norm{\partial_z^3\partial_rv^r}_0^2+\norm{\partial_z^3(\dfrac{v^r}{r})}_0^2\leq M_0+CTP\left(\sup_{t\in[0,T]} \mathfrak{E}^{\kappa}(t)\right).
	\label{divesty}
\end{equation}
Next,  we use the Proposition \ref{curlprop} and \eqref{divesty} to obtain
\begin{equation}
\norm{\partial_z^2\partial_r^2v^z}_0^2\leq M_0+CTP\left(\sup_{t\in[0,T]} \mathfrak{E}^{\kappa}(t)\right),
\end{equation}
and use the Proposition \ref{divprop} and \eqref{fefe} to obtain
\begin{equation*}
\norm{\partial_z^2\partial_r(\partial_r v^r+\dfrac{1}{r}v^r)}_0^2\leq M_0+CTP\left(\sup_{t\in[0,T]} \mathfrak{E}^{\kappa}(t)\right),
\end{equation*}
By direct calculation, we have
\begin{equation}
	\norm{\partial_z^2\partial_r^2 v^r+\partial_z^2\partial_r(\dfrac{v^r}{r})}_0^2=\norm{\partial_z^2\partial_r^2 v^r}_0^2+\norm{\partial_z^2\partial_r(\dfrac{v^r}{r})}_0^2+2\int_{\Omega}\partial_z^2\partial_r^2(r\dfrac{v^r}{r})\partial_z^2\partial_r(\dfrac{v^r}{r})\,dx
	\label{sln2}
\end{equation}
and by integration-by-parts, the last term of the above equality can be calculated as
\begin{equation}	2\int_{\Omega}\partial_z^2\partial_r^2(r\dfrac{v^r}{r})\partial_z^2\partial_r(\dfrac{v^r}{r})\,dx=3\int_{\Omega}|\partial_z^2\partial_r(\dfrac{v^r}{r})|^2\,dx+\int_{\mathbb T}\abs{\partial_z^2\partial_r(\dfrac{v^r}{r})}^2(R_0,z)R_0^2\,dz.
	\label{sfo2}
\end{equation}
Thus, we arrive at
\begin{equation}
	\norm{\partial_z^2\partial_r^2v^r}_0^2+\norm{\partial_z^2\partial_r(\dfrac{v^r}{r})}_0^2\leq M_0+CTP\left(\sup_{t\in[0,T]} \mathfrak{E}^{\kappa}(t)\right).
	\label{divesty2}
\end{equation}
Last, we can repeat the above process inductively to bound the $\norm{\partial_z\partial_r^3(v^r,v^z)}_0^2+\norm{\partial_r^4(v^r,v^z)}_0^2$ by
$M_0+TP\left(\sup\limits_{t\in[0,T]} \mathfrak{E}^{\kappa}(t)\right)$.

The estimates for $b_0\cdot\nabla R, b_0\cdot\nabla Z$ can be obtained by a similar way and the proposition is proved.
\end{proof}
\subsubsection{Synthesis}
Combining the estimates \eqref{teedde}, \eqref{00estimate} and \eqref{etaest}, we prove the Proposition \ref{th412}.

\subsection{A priori estimates for approximate system \eqref{sub2}}\label{epes}
We derive the high order energy estimates for $(v^{\theta},\Theta)$ by standard energy method.
\begin{proposition}\label{vthe}
For $t\in [0,T]$, it holds that
\begin{equation}
\label{vtheq}
\norm{v^{\theta}(t)}_4^2+\norm{ Rb_0\cdot\nabla\Theta(t)}_4^2\leq M_0+TCP\left(\sup_{t\in [0,T]}\mathfrak E(t)\right).
\end{equation}
\end{proposition}
\begin{proof}
Taking $L^2$ inner product with $v^{\theta}$ and using integration-by-parts yields
\begin{equation*}
	\begin{split}
	&\dfrac{1}{2}\dfrac{d}{dt}\int_{\Omega}|v^{\theta}|^2\,dx+\int_{\Omega} Rb_0\cdot\nabla\Theta b_0\cdot\nabla v^{\theta}\,dx\\&=\int_{\Omega}\left(-\dfrac{v^{\theta} v^r}{{R}}
+b_0\cdot\nabla {R}b_0\cdot\nabla\Theta\right) v^{\theta}\,dx
\\&=2\pi\int_{\mathbb T}\int_0^{R_0}\left(-\dfrac{ v^{\theta} v^r}{{\Rk}}
+b_0\cdot\nabla {R}b_0\cdot\nabla\Theta\right) v^{\theta}r\,drdz\\
&\leq C\norm{v^{\theta}}_0\norm{-\dfrac{ v^{\theta} v^r}{{R}}
+b_0\cdot\nabla {R}b_0\cdot\nabla\Theta}_0
\leq C(\sqrt{M})\norm{v^{\theta}}_0
\end{split}
\end{equation*}
where we used Hardy's inequality
$$\norm{\dfrac{v^{\theta}}{{R}}}_0\ls \norm{\dfrac{v^{\theta}}{r}}_0\norm{\dfrac{r}{{R}}}_{L^{\infty}(\Omega)}\ls \norm{v^{\theta}}_1, \norm{\dfrac{b_0\cdot\nabla {R}}{{R}}}_0\ls \norm{\dfrac{b_0\cdot\nabla {R}}{r}}_0\norm{\dfrac{r}{{R}}}_{L^{\infty}(\Omega)}\ls \norm{b_0\cdot\nabla {R}}_1.$$
For $\int_{\Omega}{R}b_0\cdot\nabla\Theta b_0\cdot\nabla v^{\theta}\,dx$, we have
\begin{equation}
	\begin{split}
		\int_{\Omega} Rb_0\cdot\nabla\Theta b_0\cdot\nabla v^{\theta}\,dx=&\dfrac{1}{2}\dfrac{d}{dt}\int_{\Omega}\abs{ Rb_0\cdot\nabla\Theta}^2\,dx-\int_{\Omega} R b_0\cdot\nabla\Theta \partial_t R b_0\cdot\nabla\Theta\,dx\\&+\int_{\Omega}b_0\cdot\nabla\Theta b_0\cdot\nabla R v^{\theta}\,dx\\\geq&\dfrac{1}{2}\dfrac{d}{dt}\int_{\Omega}\abs{Rb_0\cdot\nabla\Theta}^2\,dx- C(\sqrt{M})(\norm{ Rb_0\cdot\nabla\Theta}_0^2+\norm{v^{\theta}}_0^2).
\end{split}
	\label{fjiji}
\end{equation}
Hence, we have
\begin{equation*}
	\sup_{t\in [0,T]}\left(\norm{v^{\theta}}_0^2+\norm{Rb_0\cdot\nabla\Theta}_0^2\right)\leq M_0+TCP\left(\sup_{t\in [0,T]}\mathfrak{E}^{\kappa}(t)\right).
\end{equation*}
Acting $D^4$ on the second equation in \eqref{sub2}, taking $L^2$ innner product with $D^4 v^{\theta}$ and using integration-by-parts yields
\begin{equation*}
	\dfrac{1}{2}\dfrac{d}{dt}\int_{\Omega}|D^4 v^{\theta}|\,dx+\int_{\Omega}D^4(Rb_0\cdot\nabla\Theta)b_0\cdot\nabla D^4v^{\theta}\,dx=\int_{\Omega}\mathcal{G}D^4v^{\theta}\,dx,
\end{equation*}
where
\begin{equation*}
	\mathcal{G}:=[D^4, b_0\cdot\nabla]Rb_0\cdot\nabla\Theta+D^4\left(-\dfrac{v^{\theta}v^r}{{R}}+b_0\cdot\nabla {R}b_0\cdot\nabla\Theta\right).
\end{equation*}
By using the commutator estimate \eqref{co2}, we have
\begin{equation*}
	\begin{split}
\int_{\Omega}\mathcal{G}D^4 v^{\theta}\,dx&\leq \left(\norm{[D^4, b_0\cdot\nabla] Rb_0\cdot\nabla\Theta}_0+\norm{D^4\left(-\dfrac{v^{\theta}v^r}{{R}}+b_0\cdot\nabla {R}b_0\cdot\nabla\Theta\right)}_0\right)\norm{v^{\theta}}_4\\
&\leq C(\sqrt{M})P\left(\sup_{t\in [0,T]}\mathfrak{E}^{\kappa}(t)\right),
	\end{split}
\end{equation*}
where we estimate
\begin{equation*}
	\begin{split}
	\norm{D^4\left(\dfrac{v^{\theta}}{{R}} v^r\right)}_0&\leq C\left(\norm{{R} D^4(\dfrac{v^{\theta}}{{R}})}_0\abs{\dfrac{v^r}{{R}}}_{L^{\infty}}+\norm{v^r}_4\abs{\dfrac{v^{\theta}}{{R}}}_{L^{\infty}}\right)
	\\&\leq C\left(\norm{D^4v^{\theta}-\left[D^4, \dfrac{v^{\theta}}{{R}}\right]{R}}_0 \abs{\dfrac{v^r}{{R}}}_{L^{\infty}}+\norm{v^r}_4\abs{\dfrac{v^{\theta}}{{R}}}_{L^{\infty}}\right) \leq C
\end{split}
\end{equation*}
and
\begin{equation*}
	\begin{split}
		&\quad\norm{D^4(b_0\cdot\nabla{R}b_0\cdot\nabla\Theta)}_0\\&\leq C\left(\norm{{R} D^4(b_0\cdot\nabla \Theta)}_0\abs{\dfrac{b_0\cdot\nabla {R}}{{R}}}_{L^{\infty}}+\norm{b_0\cdot\nabla{R}}_4\abs{b_0\cdot\nabla\Theta}_{L^{\infty}}\right)
	\\&\leq C\left(\norm{D^4({R}b_0\cdot\nabla\Theta)-\left[D^4, b_0\cdot\nabla\Theta\right]{R}}_0 \abs{\dfrac{b_0\cdot\nabla{R}}{{R}}}_{L^{\infty}}+\norm{b_0\cdot\nabla{R}}_4\abs{b_0\cdot\nabla\Theta}_{L^{\infty}}\right) \\&\leq C
\end{split}
\end{equation*}
by using Hardy's inequality.

Using \eqref{rst} and \eqref{deflag}, we have
\begin{equation*}
	\begin{split}
		&\int_{\Omega}D^4(Rb_0\cdot\nabla\Theta)b_0\cdot\nabla D^4v^{\theta}\,dx\\=&\int_{\Omega}D^4( Rb_0\cdot\nabla\Theta)D^4\left(b_0\cdot\nabla({R}\partial_t\Theta)\right)\,dx\\&+\int_{\Omega}D^4( Rb_0\cdot\nabla\Theta)\left[D^4,b_0\cdot\nabla\right]v^{\theta}\,dx\\=&\dfrac{1}{2}\dfrac{d}{dt}\int_{\Omega}|D^4(R b_0\cdot\nabla\Theta)|^2\,dx-\int_{\Omega}D^4(Rb_0\cdot\nabla\Theta)D^4(\partial_tR b_0\cdot\nabla\Theta)\,dx\\&+\int_{\Omega}D^4(Rb_0\cdot\nabla\Theta)\left(D^4(b_0\cdot\nabla R \dfrac{v^{\theta}}{ R})+[D^4,b_0\cdot\nabla]v^{\theta}\right)\,dx\\\geq& \dfrac{1}{2}\dfrac{d}{dt}\int_{\Omega}|D^4(Rb_0\cdot\nabla\Theta)|^2\,dx - C(\sqrt M)P\left(\sup_{t\in [0,T]}\mathfrak{E}^{\kappa}(t)\right)
\end{split}
\end{equation*}
Hence, we arrive at the conclusion of this proposition.
\end{proof}
\subsection{Proof of Theorem \ref{th43}}

We now collect the estimates derived previously to conclude our estimates and also verify the a priori assumptions \eqref{ini2} and \eqref{inin3}. That is, we shall now present the
\begin{proof}[Proof of Theorem \ref{th43}]
Combining the Proposition \ref{th412} and the Proposition \ref{vthe}, we get that
\begin{equation*}
\sup_{[0,T]}\mathfrak{E}^{\kappa}(t)\leq M_0+TP\left(\sup_{t\in[0,T]} \mathfrak{E}^{\kappa}(t)\right).
\end{equation*}
This provides us with a time of existence $T_1$ independent of $\kappa$ and an estimate on $[0,T_1]$ independent of $\kappa$ of the type:
\begin{equation}
\sup_{[0,T_1]}\mathfrak{E}^{\kappa}(t)\leq 2M_0.
\end{equation}
Since by \eqref{taylor}, $-\partial_r\left(q_0-\hal\frac{C(0)^2}{r^2}\right)\ge \lambda$ on $\Gamma$, $\a^\kappa(0)=I$ and $J^\kappa(0)=1$, the bound \eqref{bound} and \eqref{pressure} verify in turn the a priori bounds \eqref{ini2} and \eqref{inin3} by the fundamental theorem of calculous with taking $T_1$ smaller if necessary.
The proof of Theorem \ref{th43} is thus completed.
\end{proof}
\section{A priori estimates for $\kappa$-approximate system with the condition \eqref{ts}.}
In this section, we derive the a priori estimate for the approximate system with the condition \eqref{ts}. The procedure is almost the same as we did in Section \ref{ae1}, and we can have the following theorem:
\begin{theorem} \label{th432}
There exists a time $T_1$ independent of $\kappa$ such that
\begin{equation}
\label{bound2}
\sup_{[0,T_1]}\mathfrak{E}^{\kappa}(t)\leq 2M_0,
\end{equation}
where $M_0=P\left(\norm{\left(v_0^r, v_0^{\theta}, v_0^z\right)}_4^2+\norm{\left(b_0^r, b_0^{\theta}, b_0^z\right)}_4^2\right).$
\end{theorem}
The only difference occurs in the derivation of the high order tangential energy estimate (Proposition \ref{te}), and we just explain how to deal with the difference in this subsection. That is, we can have the high order tangential energy estimate:
\begin{proposition}\label{te2}
Under the condition \eqref{ts}, for $t\in [0,T]$, it holds that
\begin{equation}
\label{teee2}
\norm{\partial_z^4 \nu}_0^2+\norm{\partial_z^4(b_0\cdot\nabla \zeta)(t)}_0^2+\abs{\bp^4\Lambda_{\kappa}\zeta_i \a^{\kappa}_{i1}(t)}_{L^2(\gamma')}^2\leq M_0+CTP\left(\sup_{t\in[0,T]}\mathfrak{E}^{\kappa}(t)\right).
\end{equation}
\end{proposition}
\begin{proof}
The proof of this Proposition is quite similar as the proof of the Proposition \ref{te}. The only difference is the estimate of term $\mathfrak I$ of \eqref{gg2}. Recalling the equality for $\mathfrak J$:
\begin{equation}
\begin{split}
\mathfrak I&=\int_{\mathbb T} R_0(\bp^4q-\bp^4\Lambda_\kappa^2\zeta_j\ak_{j2}\partial_z q)\ak_{i1}\mathcal{V}_i\,dz-\int_{\bT}R_0\bp^4\Lambda_\kappa^2\zeta_j \ak_{j1}\partial_r q \ak_{i1}\mathcal{V}_i\,dz
\\&= \int_{\mathbb T} -R_0\dfrac{C(t)^2}{{\Rk}^3}\left(\bp^4\Lambda_\kappa^2 R-\bp^4\Lambda_\kappa^2\zeta_j\ak_{j2}\partial_z \Lambda_\kappa^2 R\right)\ak_{i1}\mathcal{V}_i\,dz-\int_{\bT}R_0\bp^4\Lambda_\kappa^2\zeta_j \ak_{j1}\partial_r q \ak_{i1}\mathcal{V}_i\,dz\\&\quad+\int_{\bT}R_0C(t)^2\left[\bp^3, \dfrac{1}{{\Rk}^3}\right]\bp \Lambda_{\kappa}^2 R\ak_{i1}\mathcal{V}_i\,dz
\end{split}
\end{equation}
Since the Rayleigh-Taylor sign condition is satisfied only on $\gamma$, we need to depart the domain $\bT$ into two parts and deal the integral on these two parts differently. We recall the definition of $\gamma$:
$$\gamma=\{(R_0, z)| b_0^z(R_0,z)=0\}.$$
And we can have an open set $\Gamma\supset\gamma'\supset \gamma$ such that
$$\partial_r\left(q_0-\dfrac{1}{2}\abs{\dfrac{C(0)}{r}}^2\right) \leq -\dfrac{\lambda}{2}<0\,\,\text{on}\,\,\gamma'.$$
Then, by following a similar argument from \eqref{req}--\eqref{ggiip}, we can have
\begin{equation}
\begin{split}
&\int_{\gamma'}(\bp^4q-\bp^4\Lambda_\kappa^2\zeta_j\ak_{j2}\partial_z q)\ak_{i1}\mathcal{V}_i-\int_{\gamma'}\bp^4\Lambda_\kappa^2\zeta_j \ak_{j1}\partial_r q \ak_{i1}\mathcal{V}_i
\\&\geq \dfrac{d}{dt}\int_{\gamma'}\partial_r\left(q_0-\dfrac{1}{2}\abs{\dfrac{C(0)}{r}}^2\right) \abs{\bp^4\Lambda_\kappa\zeta_i \ak_{i1}}^2-P\left(\sup_{t\in[0,T]}\mathfrak{E}(t)\right)
\end{split}
\label{gest133}
\end{equation}

On the other hand, $\Gamma\backslash\gamma'$ is a compact set, and we have
\begin{equation}
\label{nonco}
\abs{b_0^z}\geq \delta>0, \text{on}\,\, \Gamma\backslash\gamma',
\end{equation}
where the $\delta$ is a constant depending on $\gamma'$.
The inequality \eqref{nonco} means the non-collinearity condition holds on $\Gamma\backslash\gamma'$. Thus, we can use it to improve the regularity of $\zeta$ on $\Gamma\backslash\gamma'$. That is
\begin{equation}
\label{gest144}
\begin{split}
&\int_{\Gamma\backslash\gamma'}(\bp^4q-\bp^4\Lambda_\kappa^2\zeta_j\ak_{j2}\partial_z q)\ak_{i1}\mathcal{V}_i-\int_{\Gamma\backslash\gamma'}\bp^4\Lambda_\kappa^2\zeta_j \ak_{j1}\partial_r q \ak_{i1}\mathcal{V}_i
\\&\le \abs{\ak Dq}_{L^{\infty}}\abs{\mathcal{V}}_{H^{-1/2}(\Gamma\backslash\gamma')}\abs{\bp^4\Lambda_\kappa^2\zeta}_{H^{1/2}(\Gamma\backslash\gamma')}
\\&\le \abs{\ak Dq}_{L^{\infty}}\abs{\mathcal{V}}_{H^{-1/2}(\Gamma\backslash\gamma')}\abs{\bp^3\left(\dfrac{b_0^z\bp\Lambda_\kappa^2\zeta}{b_0^z}\right)}_{H^{1/2}(\Gamma\backslash\gamma')} \\&\le C(\delta, \norm{b_0}_4, \norm{\ak}_3, \norm{q}_4)\norm{\nu}_4\norm{b_0\cdot\nabla\zeta}_4
\\&\le P\left(\sup_{t\in[0,T]}\mathfrak{E}(t)\right).
\end{split}
\end{equation}
Combining \eqref{gest133} and \eqref{gest144}, we can obtain the estimate for the term $\mathfrak I$ and then prove the proposition.
\end{proof}

\section{Local well-posedness of \eqref{eq:mhd}}
\begin{proof}[Proof of Theorem \ref{mainthm}]
For each $\kappa>0$, we can construct the solutions to the $\kappa$-approximate system \eqref{sub1} and \eqref{sub2} by a similar way in \cite[Section 5]{GW_16}. Briefly, we linearized the $\kappa$-approximate system and solve the linearized system by an aritifial viscosity method. Then a contract map method tells the existence of solutions to $\kappa$-approximate system \eqref{sub1} and \eqref{sub2}. We omit the details here. (The reader can also refer to \cite{Gu_17}.) Then we recover the dependence of the solutions to the $\kappa$-approximate problem \eqref{sub1} on  and \eqref{sub2} as $(v^r(\kappa), v^{\theta}(\kappa), v^z(\kappa), q(\kappa), R(\kappa), \Theta(\kappa), Z(\kappa))$.
The $\kappa$-independent estimates \eqref{bound} imply that  $(v^r(\kappa), v^{\theta}(\kappa), v^z(\kappa), q(\kappa), R(\kappa), \Theta(\kappa), Z(\kappa))$ is indeed a solution of \eqref{sub1} and \eqref{sub2} on the time interval $[0,T_1]$ and yield a strong convergence of $(v^r(\kappa), v^{\theta}(\kappa), v^z(\kappa), q(\kappa), R(\kappa), \Theta(\kappa), Z(\kappa))$ to a limit $(v^r, v^{\theta}, v^z, q, R, \Theta, Z)$, up to extraction of a subsequence, which is more than sufficient for us to pass to the limit as $\kappa\rightarrow0$ in \eqref{sub1} and \eqref{sub2} for each $t\in[0,T_1]$. We then find that $(v^r, v^{\theta}, v^z, q, R, \Theta, Z)$ is a strong solution to \eqref{eq:mhd} on $[0,T_1]$ and satisfies the estimates \eqref{enesti}. This shows the existence of solutions to \eqref{eq:mhd}.
\end{proof}
\begin{proof}[Proof of Theorem \ref{mainthm2}]
Again, for each $\kappa>0$, the construction of the solutions to the $\kappa$-approximate system \eqref{sub1} and \eqref{sub2} is just the same as we mentioned in the proof of Theorem \ref{mainthm}. Then the proof of Theorem \ref{mainthm2} is followed immediately by the $\kappa$-independent estimates \eqref{bound2}.
\end{proof}
\section*{Acknowledgement:}

 The author was in supported by NSFC (grant No. 11601305).

\end{document}